\documentclass[12pt]{amsart} 

\usepackage[margin=3cm]{geometry}

\usepackage[dvipsnames]{xcolor}
\usepackage{tikz-cd}
\usetikzlibrary{shapes.geometric}
\usetikzlibrary{decorations.pathmorphing,shapes}
\usepackage[hidelinks]{hyperref}
\usepackage{ amssymb, amsthm, amsmath, bbm }
\usepackage{mathtools}
\usepackage[ruled,vlined]{algorithm2e}

\newtheorem{theorem}{Theorem}[section]
\newtheorem{corollary}[theorem]{Corollary} 
\newtheorem{lemma}[theorem]{Lemma}
\newtheorem{proposition}[theorem]{Proposition}

\newtheorem{problem}[theorem]{Problem}

\newtheorem{remark}[theorem]{Remark}

\theoremstyle{definition}
\newtheorem{example}[theorem]{Example}
\newtheorem{definition}[theorem]{Definition}

\newcommand{\tr}{\mathrm{tr}}

\newcommand{\CC}{{\mathbb{C}}}
\newcommand{\RR}{{\mathbb{R}}}

\newcommand{\ZZ}{{\mathbb{Z}}}

\newcommand{\Mcal}{{\mathcal{M}}}
\newcommand{\fX}{{\mathfrak{X}}}
\newcommand{\T}{^\mathsf{T}}

\newcommand{\SL}{\operatorname{SL}}
\newcommand{\GL}{\operatorname{GL}}

\newcommand{\pa}{\mathrm{pa}}
\newcommand{\ch}{\mathrm{ch}}

\newcommand{\PD}{{\rm PD}}
\newcommand{\prc}{{\rm prc}}
\newcommand{\Gcal}{{\mathcal G}}
\newcommand{\Ccal}{{\mathcal C}}

\newcommand{\mlt}{{\rm mlt}}
\newcommand{\id}{{\rm id}}
\newcommand{\rk}{{\rm rk}}

\newcommand*\circled[1]{\tikz[baseline=(char.base)]{\node[shape=circle,draw,inner sep=1pt] (char) {#1};}}
\newcommand*\squared[1]{\tikz[baseline=(char.base)]{\node[shape=rectangle,draw,inner sep=1.5pt] (char) {#1};}}
\newcommand*\triangled[1]{\tikz[baseline=(char.base)]{\node[shape=regular polygon,draw,regular polygon sides = 3,inner sep=-0.3pt] (char) {#1};}}
\newcommand*\pentagoned[1]{\tikz[baseline=(char.base)]{\node[shape=regular polygon,draw,regular polygon sides = 5,inner sep=-0.1pt] (char) {#1};}}

\title{Symmetries in directed Gaussian graphical models}
\author{Visu Makam, Philipp Reichenbach, and Anna Seigal}

\begin{document}

\maketitle

\begin{abstract}
We define Gaussian graphical models on directed acyclic graphs with coloured vertices and edges, calling them RDAG (restricted directed acyclic graph) models.
If two vertices or edges have the same colour, their parameters in the model must be the same. 
We present an algorithm to find the maximum likelihood estimate (MLE) in an RDAG model, and characterise when the MLE exists, via linear independence conditions. We relate properties of a graph, and its colouring, to the number of samples needed for the MLE to exist and to be unique. 
We also characterise when an RDAG model is equal to an associated undirected graphical model and study connections to groups and invariant theory. We provide examples and simulations to study the benefits of RDAGs over uncoloured DAGs.
\end{abstract}

\section*{Introduction}

The concept of a graph is widely used across the sciences~\cite{foulds2012graph}.
Graphs are a framework to relate entities: the vertices are the entities of interest, and the edges encode connections between them. A graph is given a statistical meaning in the study of graphical models~\cite{lauritzen1996graphical,maathuis2018handbook}. 
Each vertex represents a random variable, and the edges between variables reflect their statistical dependence~\cite{verma1990causal}.
In this paper, we study {\em directed} Gaussian graphical models, also called Gaussian Bayesian networks, or linear structural equation models with independent errors~\cite{Sullivant}. 
Such models have been applied to cell signalling~\cite{sachs2005causal}, gene interactions~\cite{friedman2000using}, causal inference~\cite{pearl2009causality}, and many other contexts.

We define graphical models on directed acyclic graphs (DAGs) with a colouring of their vertices and edges. Vertices or edges with the same colour must have the same parameter values. Thus, the graph colouring imposes symmetries in the model.
We call such models {\em RDAG models}, where the `R' stands for restricted, cf.~\cite{hojsgaard2008graphical}.

Our first motivation for RDAG models is that vertex and edge symmetries appear in various applications, such as in the study of longitudinal data~\cite{abbruzzo2016operational,vinciotti2016model}, or clustered variables~\cite{gao2015estimation,hojsgaard2008graphical}.
The coloured directed graph gives an intuitive pictorial description of the symmetry conditions in the model.

Our second motivation is to decrease the maximum likelihood threshold, the minimum sample size required for the maximum likelihood estimate (MLE) to exist almost surely, see~\cite{dempster1972covariance}. In applications, it is desirable for the MLE to exist when there is only a small number of samples; i.e., for the maximum likelihood threshold to be small.
Innovative ideas have been used to find maximum likelihood thresholds in graphical models~\cite{buhl1993existence,drton2019maximum,gross2018maximum,uhler2012geometry} and for estimating the MLE from too few samples~\cite{friedman2008sparse,wille2004sparse}. 
Removing edges from a graph can lower the threshold~\cite{uhler2012geometry, lauritzen1996graphical}, but there is a trade-off: removing edges imposes more conditional independence among the variables. This is why, instead, we aim to decrease the maximum likelihood threshold by introducing symmetries. 

We will use the following as our running example throughout the paper.

\begin{example}
\label{ex:very_first}
Consider the coloured graph \begin{tikzcd}[cramped, sep = small]
        {\color{blue}\circled{1}} & \squared{3} \ar[r, red] \ar[l, red] & {\color{blue}\circled{2}}
    \end{tikzcd}, with blue (circular) vertices $\{ 1, 2\}$, black (square) vertex $3$ and two red edges.
The RDAG model is
\[ y_1 = \lambda y_3 + \epsilon_1 , \qquad y_2 = \lambda y_3 + \epsilon_2, \qquad y_3 = \epsilon_3, \]
where $\epsilon_1, \epsilon_2 \sim N(0,\omega)$ and $\epsilon_3 \sim N(0,\omega')$, i.e. $\omega$ is the variance of blue vertices $1$ and $2$ and $\omega'$ is the variance of black vertex $3$. The third parameter $\lambda$ is the regression coefficient given by a red edge.
We will see that the MLE exists uniquely (almost surely) given one sample. 
For comparison, if we remove the colours the resulting model needs two samples for the MLE to exist.
We use this example to model the dependence of two daughters' heights on the height of their mother, and we compute the MLE given some sample data, in Section~\ref{sec:illustrative}.
\end{example}

As far as we are aware, RDAG models have not been defined before in the literature; we comment on some related models.
The assumption of equal variances from~\cite{peters2014identifiability} is the special case of an RDAG model, where all vertex colours are the same.
Special colourings encode exchangeability between variables, or invariance under a group of permutations.
A graphical model is combined with group symmetries in the directed setting in~\cite{madsen2000invariant} and in the undirected setting in~\cite{andersson1998symmetry,shah2012group}. 
RDAG models also relate to the fused graphical lasso~\cite{danaher2014joint}, which penalises differences between parameters on different edges, whereas in an RDAG model the parameters on edges of the same colour must be equal.

In this paper, we give a closed-form formula for the MLE in an RDAG model, as a collection of least squares estimators, see Algorithm~\ref{algorithm:the}.
We characterise the existence and uniqueness of the MLE via linear algebraic properties of the sample data, see Theorem~\ref{thm:MLestimationLinDependence}. 
We give upper and lower bounds on the threshold number of samples required for existence and uniqueness of the MLE in Theorem~\ref{thm:RDAGboundsMlt}. Our results show that RDAG thresholds are less or equal to the DAG thresholds, and that high symmetry decreases the thresholds.
Finally, we compare RDAG MLEs to uncoloured DAG MLEs via simulations in Section~\ref{sec:simulations}.
Our results hold with an assumption on the graph colouring, which we call {\em compatibility} (Defintion~\ref{dfn:compatibleColoring}). It is an open problem to extend our results to the non-compatible setting, as well as to directed graphs with~cycles. It is also an open problem to find the exact maximum likelihood thresholds, see Problem~\ref{prob:1}.

The undirected analogue to RDAG models are the RCON models from~\cite{hojsgaard2008graphical}.
Although a motivation for the graph colouring in RCON models is to lower the maximum likelihood threshold, there are relatively few graphs for which the threshold is known: colourings of the four cycle are studied in~\cite[\S 6]{uhler2012geometry},~\cite[\S 5]{sturmfels2010multivariate}, while an example with five vertices is~\cite[Example 3.2]{uhler2012geometry}.
In certain cases, RDAG models are equivalent to RCON models. We determine precisely the conditions under which this occurs in Theorem~\ref{thm:RCONequalsRDAG}. As a consequence, we obtain an entire class of RCON models where conditions for MLE existence and uniqueness can be found by appealing to our results on RDAGs.

This paper has two appendices, where we explain some connections to invariant theory. A Gaussian group model~\cite{amendola2020invariant} is parametrised by a group. 
In~\cite{amendola2020invariant}, the authors draw a dictionary between maximum likelihood estimation and stability notions in invariant theory. This dictionary allows for the transfer of tools from the algebraic subjects of representation theory and invariant theory to statistics: maximum likelihood thresholds were computed for matrix normal models in~\cite{derksen2021maximum} and for tensor normal models in~\cite{derksen2020maximum}.

We extend the dictionary between maximum likelihood estimation and stability notions to RDAGs in Theorem~\ref{thm:RDAGstabilityVsMLE}.
This requires us to extend the definitions of stability beyond the setting of a group action, see Definition~\ref{defn:stability_setE}.
While not evident in our final presentation, this perspective gave us the understanding needed to obtain many of the results in this paper and we would like to stress its importance for future work.  We have far more tools at our disposal when a model is backed by a group action, i.e., when it is a Gaussian group model. We identify RDAGs that are Gaussian group models in Proposition~\ref{prop:butterfly} and exhibit additional tools that one can use in such cases. 
The two appendices offer two alternative descriptions of the set of MLEs, see Propositions~\ref{prop:second-bijection} and~\ref{prop:StabiliserMLEsGroup}. 

\tableofcontents

\addtocontents{toc}{\protect\setcounter{tocdepth}{1}}

\section{Preliminaries}
\label{sec:preliminaries}

\subsection{Multivariate Gaussian models}
We consider $m$-dimensional Gaussian distributions with mean zero. 
Such a distribution is determined by its concentration (inverse covariance) matrix $\Psi$, a real $m \times m$ positive definite matrix. The
density function is
\begin{align*}
    f_\Psi(y) = \frac{1}{\sqrt{\det( 2 \pi \Psi^{-1})}}
    \exp \left( -\frac{1}{2} y\T \Psi y \right), \qquad y \in \RR^m. 
\end{align*}
We refer to a multivariate Gaussian model by the set of concentration matrices in the model. So, a model is a subset of $\PD_m$, the cone of $m \times m$ positive definite matrices. 

We study statistical models in $\PD_m$ via a set of invertible matrices. We define
\begin{equation}
    \label{eqn:ME}
    \Mcal_A := \{ a\T a \mid a \in A \} ,
\end{equation} 
where $A$ is a subset of $\GL_m$, the real invertible $m \times m$ matrices.
Many sets $A$ can correspond to the same model $\Mcal_A$. For instance, the full cone $\PD_m$ is $\Mcal_A$ whenever $A$ contains all invertible upper triangular matrices.  When the set $A$ is a group, the model $\Mcal_A$ is called a {\em Gaussian group model}~\cite{amendola2020invariant}.
    
\subsection{Maximum likelihood estimation}
    
    A maximum likelihood estimate (MLE) is a point in the model that maximizes the likelihood of observing some data. 
For $n$ samples from a Gaussian model $\mathcal{M} \subseteq \PD_m$, the data samples are the columns of a matrix $Y \in \RR^{m \times n}$. Assuming independent samples, 
the likelihood function is
$$ L_Y(\Psi) = \prod_{i=1}^n f_{\Psi}(Y_i),$$
where $Y_i$ is the $i$th column of $Y$. 
We work with the log-likelihood function $\log L_Y$, which has the same maximisers as $L_Y$. 
The log-likelihood function can be written, up to additive and multiplicative constants, as
\begin{equation}\label{eqn:gaussianlikelihood}
    \ell_{Y} (\Psi) = \log \det (\Psi) - \mathrm{tr} (\Psi S_Y),
\end{equation}
where $S_Y = \frac{1}{n} \sum_{i=1}^n Y_i Y_i\T$ is the sample covariance matrix.
        Four possibilities arise when maximising the log-likelihood:
\begin{enumerate}
    \item[(a)] $\ell_Y$ unbounded from above
    \item[(b)] $\ell_Y$ bounded from above
    \item[(c)] the MLE exists (i.e. $\ell_Y$ is bounded from above and attains its supremum)
    \item[(d)] the MLE exists and is unique.
\end{enumerate}
The minimal number of samples needed for the MLE to exist almost surely is the MLE existence threshold; the number of samples for the MLE to exist uniquely almost surely is the uniqueness threshold.

\begin{example}
    Let $\Mcal = \PD_m$. The unique maximiser of the likelihood is $\hat{\Psi} = S_Y^{-1}$, if $S_Y$ is invertible. If $S_Y$ is not invertible, the likelihood function is unbounded from above, see \cite[Proposition~ 5.3.7]{Sullivant}. The existence and uniqueness thresholds are therefore both $m$, since with $m$ samples the matrix $S_Y$ will almost surely be invertible.
\end{example}

For a model $\Mcal_A$ as in~\eqref{eqn:ME},
we can rewrite the log-likelihood~\eqref{eqn:gaussianlikelihood} at sample matrix $Y \in \RR^{m \times n}$ as a function of a matrix $a \in A$.
\begin{equation}
    \label{eqn:likelihood_e}
     \ell_Y(a\T a) = \log \det (a\T a) - \frac{1}{n} \Vert a \cdot Y \Vert^2.
\end{equation}

    \subsection{Directed Gaussian graphical models (DAG models)}
A directed acyclic graph (DAG) is $\Gcal = (I,E)$, where $I$ is a set of vertices, and $E$ a set of directed edges. 
We write $j \to i$ for an edge from $j$ to $i$; the absence of such an edge is denoted $j \not\to i$. The {\em parents} and {\em children} of $i$ in $\Gcal$ are, respectively, the vertex sets
$$ \pa(i) = \{ j \in I \mid (j \to i) \in E \} \qquad \ch(i) = \{ k \in I \mid (i \to k) \in E \}.$$
We often take the vertex set $I$ to be $[m] = \{ 1, 2, \ldots, m\}$. 

We call a directed Gaussian graphical model on $\Gcal$ a {\em DAG model}. A DAG model is defined by the linear structural equation
\begin{equation}
    \label{eqn:lsem}
    y = \Lambda y + \varepsilon, \qquad \text{i.e.} \qquad y_i = \sum_{j \in \pa(i)} \lambda_{ij} y_j + \varepsilon_i,
\end{equation} 
where $y \in \RR^m$, $\lambda_{ij}=0$ for $j \not\to i$ in $\mathcal{G}$, and $\epsilon \sim N(0,\Omega)$ with $\Omega$ diagonal.
The coefficient $\lambda_{ij}$ is a regression coefficient, the effect of parent $j$ on child~$i$.
The model encodes conditional independence: a node is independent of its non-descendants, after conditioning on its parents~\cite{verma1990causal}. 

\begin{remark}[The matrix $\Lambda$ is strictly upper triangular]
Throughout the paper, we choose an ordering on the vertices of $\Gcal$ so that $\Lambda$ is upper triangular.  That is, if edge $j \to i$ is in $E$ then $j > i$.
Such an ordering is possible because~$\Gcal$ is acyclic. Thinking of a vertex label as its age, the ordering makes parents older than children.
\end{remark}

Solving~\eqref{eqn:lsem} for $y$ gives $y = (\id - \Lambda)^{-1} \varepsilon$, 
where $\id$ denotes the $m \times m$ identity matrix, and the acyclicity of $\mathcal{G}$ ensures that $(\id - \Lambda)$ is invertible.
Hence $y$ is multivariate Gaussian with mean zero and concentration matrix
\begin{equation}
    \label{eq:graphmodel}
\Psi = (\id - \Lambda)\T \Omega^{-1} (\id-\Lambda).
\end{equation}
We define a set of matrices associated to the DAG $\Gcal$
\begin{equation}
\label{eq:AG}
 A(\mathcal{G}) = \{ a \in \GL_m \mid a_{ij}=0 \text{ for } i \neq j \text{ with } j \not \to i \text{ in }  \mathcal{G} \} .
 \end{equation}
 Recall from~\eqref{eqn:ME} the notation 
$\Mcal_{A(\Gcal)} = \{ a\T a : a \in A(\Gcal) \}$.
The set of concentration matrices of the form~\eqref{eq:graphmodel} is equal to the set $\Mcal_{A(\Gcal)}$. We prove this in Lemma~\ref{lem:DAG_MA}.

\subsection{Undirected Gaussian graphical models}

Multivariate Gaussian models can also be obtained from undirected graphs. An undirected graph $\Gcal = (I,E)$ is a set of vertices $I$ and \emph{undirected} edges $E$. The model is the set of distributions with mean zero and concentration matrix $\Psi$ with $\Psi_{ij} = 0$ whenever edge \begin{tikzcd}[cramped, sep=small]
i \ar[r, no head] & j
\end{tikzcd}
is not in $E$. That is, the variables at nodes $i$ and $j$ are independent after conditioning on all others, see ~\cite[Proposition~13.1.5]{Sullivant}.

\subsection{Restricted concentration (RCON) models}

In~\cite{hojsgaard2008graphical}, the authors introduce restricted concentration (RCON) models, which impose symmetries on the concentration matrix $\Psi$ according to a graph colouring.
A coloured undirected graph is a tuple $(\mathcal{G},c)$, where $\mathcal{G} = (I, E)$ is an undirected graph and the map
$$ c: I \cup E \rightarrow \mathcal{C}$$ 
assigns a colour to each vertex and to each edge. The vertex $i \in I$ has colour $c(i) \in \Ccal$, and edge
\begin{tikzcd}[cramped, sep=small]
i \ar[r, no head] & j
\end{tikzcd}
has colour $c(ij) \in \Ccal$.

\begin{definition}[{see~\cite[\S3]{hojsgaard2008graphical}}]
\label{def:rcon}
The {\em RCON model} on the coloured undirected graph $(\Gcal, c)$ consists of concentration matrices with
\begin{enumerate}
    \item $\Psi_{ij} = 0$ whenever \begin{tikzcd}[cramped, sep=small]
    i \ar[r, no head] & j
    \end{tikzcd} is not in $E$
    \item $\Psi_{ii} = \Psi_{jj}$ whenever $c(i) = c(j)$,
    \item $\Psi_{ij} = \Psi_{kl}$ whenever $c(ij) = c(kl)$.
\end{enumerate}
\end{definition}

\section{Introducing RDAG models}

\label{sec:RDAG_def}

A colouring of a DAG assigns colours to the vertices and edges. A coloured DAG is a tuple $(\mathcal{G},c)$, where $\mathcal{G} = (I, E)$ is a DAG on vertices $I$ and edges $E$, and
$$ c: I \cup E \rightarrow \mathcal{C}$$ 
is a colouring of the vertices and edges. Vertex $i \in I$ has colour $c(i) \in \Ccal$, and edge $j \to i$ has colour $c(ij) \in \Ccal$. We sometimes denote the vertex colour $c(i)$ by $c(ii)$, with no ambiguity because a DAG cannot have loops.

\begin{definition}\label{dfn:RDAGmodelViaLDL}
The {\em restricted DAG (RDAG) model} on the coloured DAG $(\Gcal,c)$ is
the set of concentration matrices $\Psi = (\id- \Lambda)\T \Omega^{-1} (\id - \Lambda)$,
where $\Lambda \in \RR^{m \times m}$ satisfies
\begin{enumerate}
\item $\lambda_{ij} = 0$ unless $j \to i$ in $\Gcal$
\item $\lambda_{ij} = \lambda_{kl}$ whenever edges $j \to i$ and $l \to k$ have the same colour
\end{enumerate}
and the diagonal matrix $\Omega \in \RR^{m \times m}$ has positive entries and satisfies
\begin{enumerate}
\item[(3)] $\omega_{ii} = \omega_{jj}$ if vertices $i$ and $j$ have the same colour.
\end{enumerate}
The model is given by the linear structural equation
$y = \Lambda y + \varepsilon$,
where $\varepsilon \sim N(0,\Omega)$.
\end{definition}

\begin{example}
\label{ex:RDAGminus1}
Consider the coloured graph \begin{tikzcd}[cramped, sep = small]
        {\color{blue}\circled{1}} & \squared{3} \ar[r, red] \ar[l, red] & {\color{blue}\circled{2}}
    \end{tikzcd} from Example~\ref{ex:very_first}.
The RDAG model 
is parametrised by matrices \[ 
\Lambda \in \left\lbrace \begin{pmatrix} 0 & 0 & \lambda \\ 0 & 0 & \lambda \\ 0 & 0 & 0 \end{pmatrix} : \lambda \in \RR \right\rbrace 
\qquad \text{and} \qquad 
\Omega
 \in \left\lbrace \begin{pmatrix} \omega & 0 & 0 \\ 0 & \omega & 0 \\ 0 & 0 & \omega' \end{pmatrix} : \omega, \omega' > 0 \right\rbrace. 
\] 
\end{example}

We will parametrise the RDAG model on $(\mathcal{G},c)$ via the set
 \begin{equation}
     \label{eqn:Agc}
     A(\mathcal{G},c) := \left\lbrace a \in \GL_m \bigg| \; 
     \begin{matrix} a_{ij}=0 \text{ for } i \neq j \text{ with } j \not \to i \text{ in }  \mathcal{G} \\ 
     a_{ij} = a_{kl} \text{ whenever }  c(ij) = c(kl) \end{matrix} \right\rbrace.
 \end{equation}
Note that $A(\Gcal, c)$ is contained in the set $A(\Gcal)$ from~\eqref{eq:AG}: the zero patterns of $A(\Gcal)$ and $A(\Gcal,c)$ are the same, and $A(\Gcal,c)$ has further equalities imposed by the colouring $c$.
Before we characterise which RDAG models can be parametrised by $A(\Gcal,c)$, we pause to motivate the use of this alternative parametrisation.

\begin{remark}[Motivation for the parametrisation via $A(\Gcal,c)$]
The RDAG model on $(\Gcal,c)$ will be the set $\Mcal_A$ as in~\eqref{eqn:ME}, where $A := A(\Gcal,c)$.
This point of view is motivated by connections to invariant theory for transitive DAG models in~ \cite[Section~5]{amendola2020invariant}. The alternative parametrisation has useful consequences. First, it leads to a condition on the graph colouring, called compatibility, which is indispensable in our results of Sections~\ref{sec:MLE_existence} and \ref{sec:thresholds}. Second, it is helpful when comparing directed and undirected models in Section~\ref{sec:RDAG_vs_RCON}. Finally, it enables us to generalise the connections to invariant theory from \cite{amendola2020invariant} to the setting of RDAGs, see Appendices~\ref{sec:appendix_stability} and~\ref{sec:connections_GGM}.
\end{remark}

\begin{example}
\label{ex:RDAG0}
Returning to the example 
\begin{tikzcd}[cramped, sep = small]
        {\color{blue}\circled{1}} & \squared{3} \ar[r, red] \ar[l, red] & {\color{blue}\circled{2}}
    \end{tikzcd}, we have
\[ A(\Gcal,c) = \left\lbrace \begin{pmatrix} d_1 & 0 & r \\ 0 & d_1 & r \\ 0 & 0 & d_2 \end{pmatrix} \colon d_1, d_2 \neq 0, \; r \in \RR \right\rbrace .\]
\end{example}

We now introduce a natural assumption on a colouring. 

\begin{definition}\label{dfn:compatibleColoring}
A colouring $c$ of a directed graph is \emph{compatible}, if:
    \begin{itemize}
        \item[(i)] Vertex and edge colours are disjoint; and
        \item[(ii)] If edges $j \to i$ and $l \to k$ have the same colour, then the child vertices $i$ and $k$ also have the same colour, i.e. $c(ij) = c(kl) \implies c(i) = c(k)$.
    \end{itemize}
    Note that compatibility does \emph{not} impose equality of parent colours $c(j)$ and $c(l)$. 
\end{definition}

\begin{remark}[Motivation for compatibility]
In an RDAG model, we do not impose equalities between $\Omega$ and $\Lambda$.
The entry $\omega_{ii}$ is a variance, while $\lambda_{kl}$ is a regression coefficient, so setting them to be equal would be difficult to interpret.
Hence the vertex and edge colours can always be thought of as disjoint, 
as in compatibility condition (i).
Compatibility condition (ii) has the statistical interpretation that the same regression coefficient appearing in an expression for two variables implies that their error variances agree.
This extra assumption is indispensable in many of our results and proofs.
It is a directed analogue to the condition appearing in~\cite[Proposition 1]{hojsgaard2008graphical}.
\end{remark}

The first use of compatibility condition (ii) is in relating an RDAG model on $(\Gcal,c)$ to the set $A(\Gcal,c)$. 
As in~\eqref{eqn:ME}, we consider the model
$$ \Mcal_{A(\Gcal,c)} = \big\lbrace a\T a \mid a \in A(\Gcal,c) \big\rbrace . $$

\begin{proposition}\label{prop:compatibleColouring}
Fix a coloured DAG $(\Gcal,c)$. The RDAG model on $(\Gcal,c)$ is equal to $\Mcal_{A(\Gcal,c)}$ if and only if the colouring $c$ is compatible.
\end{proposition}

Before proving the proposition, we recall two matrix decompositions. The LDL decomposition writes a positive definite matrix as $\Psi = LDL\T$, where $D$ is diagonal with positive entries, and $L$ is lower triangular and unipotent (i.e. has ones on the diagonal). 
The LDL decomposition is closely related to the factorisation $\Psi = (\id - \Lambda)\T \Omega^{-1} (\id - \Lambda)$ from~\eqref{eq:graphmodel}: the LDL decomposition is $D = \Omega^{-1}$ and $L = (\id - \Lambda)\T$.
Hence an RDAG model imposes zeros and symmetries in the LDL decomposition.

The second matrix decomposition is the Cholesky decomposition. It writes a postive definite matrix as the product $\Psi = a\T a$, where $a$ is upper triangular with positive diagonal entries.
The model $\Mcal_{A(\Gcal,c)}$
imposes zeros and symmetries in the Cholesky decomposition, as follows.

\begin{lemma}
\label{lem:cholesky_MA}
Fix a coloured DAG $(\Gcal,c)$ with compatible colouring $c$. Then $\Mcal_{A(\Gcal,c)}$ is the set of matrices with Cholesky decomposition $a\T a$ for some $a \in A(\Gcal,c)$.
\end{lemma}

\begin{proof} 
The set $\Mcal_{A(\Gcal,c)}$ consists of all matrices $\Psi$ of the form
$a\T a$ for some $a \in A(\Gcal,c)$, see~\eqref{eqn:ME}. The matrix $a$ is upper triangular by the structure of $\Gcal$. 
To get the Cholesky decomposition, it remains to modify~$a$ to have positive diagonal entries. 
 We replace $a$ by $ka$, where $k$ is the diagonal matrix with $k_{ii} = 1$ if $a_{ii} > 0$ and $k_{ii} = -1$ if $a_{ii} < 0$. Then $ka$ 
 flips the sign of all rows of $a$ with negative diagonal entry, hence it
 has all diagonal entries strictly positive. The compatibility of the colouring ensures that $a_{ij} = a_{kl}$ can only hold in $A(\Gcal,c)$ if $a_{ii} = a_{kk}$. Hence multiplying by $k$ doesn't break any edge compatibility conditions, and $ka \in A(\Gcal,c)$.
\end{proof}

The LDL and Cholesky decompositions are both unique, since $\Psi$ is (strictly) positive definite. They are related by:
$$ \begin{matrix*}[l] \text{Cholesky from LDL:} & a = D^{1/2} L\T, & \\ 
\text{LDL from Cholesky:} & D = \mathrm{diag}(a_{11}^2,\ldots,a_{mm}^2), & L\T = D^{-1/2} a.\end{matrix*} $$ 
The following lemma is proved by comparing zero patterns in the two decompositions. 

\begin{lemma}
\label{lem:DAG_MA}
The DAG model on $\Gcal$ is the model~$\Mcal_{A(\Gcal)}$.\end{lemma}

\begin{proof}
The LDL decomposition of $\Psi = (\id - \Lambda)\T \Omega^{-1} (\id - \Lambda)$ is given by $D = \Omega^{-1}$ and $L = (\id - \Lambda)\T$. The Cholesky decomposition has 
\begin{equation}
    \label{eqn:entries_a}
     a = D^{1/2} L\T = \Omega^{-1/2} (\id - \Lambda), \quad \text{i.e.} \quad a_{ij} = \begin{cases} \omega_{ii}^{-1/2} & \text{if } i = j \\ 
-\omega_{ii}^{-1/2} \lambda_{ij} 
& \text{if } i \neq j .
\end{cases} 
\end{equation}
We have containment of the DAG model inside $\Mcal_{A(\Gcal)}$, because
if $j \not \to i$ in $\Gcal$, then $\lambda_{ij} = 0$ and therefore $a_{ij} = 0$. Conversely, given $\Psi \in \Mcal_{A(\Gcal)}$, its Cholesky decomposition is $a\T a$ for some $a \in A(\Gcal)$ by Lemma~\ref{lem:cholesky_MA}, since the colouring that assigns all vertices and edges different colours satisfies $A(\Gcal) = A(\Gcal,c)$ and is compatible. Hence $a_{ij} = 0$ for $i \neq j$ unless $j \to i$ in~$\Gcal$. We therefore have $(\id - \Lambda)_{ij} = L_{ji} = \frac{1}{a_{ii}} a_{ij} = 0$, where $a_{ii} \neq 0$, and hence $\lambda_{ij} = 0$. 
\end{proof}

We prove Proposition~\ref{prop:compatibleColouring}, by comparing symmetries in the two decompositions.

\begin{proof}[Proof of Proposition~\ref{prop:compatibleColouring}]
Given $\Psi$ in the RDAG model, we show that its Cholesky decomposition is $\Psi = a\T a$ with $a \in A(\Gcal,c)$. By Lemma~\ref{lem:DAG_MA}, $a \in A(\Gcal)$ and it remains to show that the colour conditions hold, i.e. that $a_{ij} = a_{kl}$ whenever $c(ij) = c(kl)$. 
If $c(ii) = c(kk)$, then $\omega_{ii} = \omega_{kk}$ since $\Omega$ respects the colouring.
This shows that $a_{ii} = a_{kk}$, using~\eqref{eqn:entries_a}. 
If $c(ij) = c(kl)$ for edges $j \to i$ and $l \to k$, then $\lambda_{ij} = \lambda_{kl}$ since $\Lambda$ respects the colouring and, moreover, $\omega_{ii} = \omega_{kk}$ by compatibility. This implies $a_{ij} = a_{kl}$, by~\eqref{eqn:entries_a}.

Conversely, given $\Psi \in \Mcal_{A(\Gcal,c)}$, we show that $\Psi$ is in the RDAG model. The Cholesky decomposition is $\Psi = a\T a$ for $a \in A(\Gcal,c)$, by Lemma~\ref{lem:cholesky_MA}. The entries of $\Omega$ and $\Lambda$ are $\omega_{ii} = a_{ii}^{-2}$ and $\lambda_{ij} = -a_{ii}^{-1} a_{ij}$, by~\eqref{eqn:entries_a}, which satisfy the RDAG model conditions. Hence a compatible colouring implies the equivalence of the RDAG model on $(\Gcal,c)$ and $\Mcal_{A(\Gcal,c)}$.

If the colouring is not compatible, we exhibit some $\Psi$ in the RDAG model that is not in $\Mcal_{A(\Gcal,c)}$. Let $\Psi = a\T a$ be the Cholesky decomposition. If there is some $a' \in A(\Gcal,c)$ with $\Psi = (a')\T a'$ then, similar to the proof of Lemma~\ref{lem:cholesky_MA}, there is a diagonal matrix $o$ with entries $\pm 1$ with $oa' = a$. First, if Definition~\ref{dfn:compatibleColoring}(i) does not hold, there is a vertex $k$ and an edge $j \to i$ with $c(k) = c(ij)$. The RDAG model imposes no relation between $\omega_{kk}$ and $\lambda_{ij}$, so let $\Psi$ be given by some $\Omega$ and $\Lambda$ with $\omega_{kk} = 1$ and $\lambda_{ij} = 0$. Then $o_{kk} a'_{kk} = a_{kk} = 1$ and $o_{ii} a'_{ij} = a_{ij} = 0$, by~\eqref{eqn:entries_a}. Hence, $a'_{kk} \neq 0 = a'_{ij}$ and therefore $\Psi \notin \Mcal_{A(\Gcal,c)}$.
Second, if Definition~\ref{dfn:compatibleColoring}(ii) does not hold, then there exist edges $j \to i$ and $l \to k$ with $c(ij) = c(kl)$ but $c(i) \neq c(k)$. We choose $\Psi$ given by some $\Omega$ and $\Lambda$ with $\omega_{ii} = 1$, $\omega_{kk} = \frac14$ and $\lambda_{ij} = \lambda_{kl} = -1$. Then $o_{ii} a'_{ij} = a_{ij} = 1$ and $o_{kk} a'_{kl} = a_{kl} = 2$, by \eqref{eqn:entries_a}. Thus, $\vert a'_{ij} \vert  = 1 \neq 2 = \vert a'_{kl} \vert$ and, again, we cannot have $\Psi \in \Mcal_{A(\Gcal,c)}$.
\end{proof}

\begin{example}
\label{ex:RDAG1}
We return to the graph
\begin{tikzcd}[cramped, sep = small]
        {\color{blue}\circled{1}} & \squared{3} \ar[r, red] \ar[l, red] & {\color{blue}\circled{2}}
    \end{tikzcd}
from
Examples~\ref{ex:RDAGminus1} and~\ref{ex:RDAG0}.
The colouring is compatible, because the set of vertex colours $\{ \text{blue}, \text{black} \}$ is disjoint from the edge colour set $\{ \text{red} \}$, and the children of both red edges have the same colour. Hence Proposition~\ref{prop:compatibleColouring} shows that the RDAG model is equal to     \begin{equation}
    \label{eqn:RDAG1}
        \Mcal_{A(\Gcal,c)}=
        \left\lbrace \begin{pmatrix} d_1^2 & 0 & r d_1 \\ 0 & d_1^2 & r d_1 \\ r d_1 & r d_1 & 2r^2 + d_2^2 \end{pmatrix} \, \Bigg\vert \, d_1, d_2 \neq 0 \right\rbrace.
    \end{equation}
\end{example}

\begin{remark}\label{rem:complexRDAG}
RDAG models can also be defined over the complex numbers. Here, the parameters $\Lambda$ can be complex, and we obtain a subset of $\PD_m$ by taking conjugate transposes, $\Psi = (\id - \Lambda)^\dagger \Omega^{-1} (\id - \Lambda)$. For the $\Mcal_A$ characterisation, we replace $a\T a$ by $a^\dagger a$. Many of our results and proofs can be modified to hold in the complex setting. We return to complex RDAGs in Section~\ref{sec:popov}.
\end{remark}

\section{Comparison of RDAG and RCON models}
\label{sec:RDAG_vs_RCON}

Given a directed graph, we can forget the direction of each edge to give an undirected graph. We characterise when the RDAG model on the coloured directed graph is equal to the undirected model on the corresponding coloured undirected graph in Theorem~\ref{thm:RCONequalsRDAG}.
We begin by comparing RDAG and RCON models in two examples.

\begin{example}[RDAG $=$ RCON]\label{ex:RCONequalsColoredTDAG}
We revisit our running example  \begin{tikzcd}[cramped, sep = small]
        {\color{blue}\circled{1}} & \squared{3} \ar[r, red] \ar[l, red] & {\color{blue}\circled{2}}
    \end{tikzcd}. 
The corresponding RCON model has coloured undirected graph \begin{tikzcd}[cramped, sep = small]
        {\color{blue}\circled{1}} & \squared{3} \ar[r, red, no head] \ar[l, red, no head] & {\color{blue}\circled{2}}
    \end{tikzcd}, with blue (circular) vertices $\{ 1, 2\}$, black (square) vertex $3$, and red edges. 
By Definition~\ref{def:rcon}, the RCON model is the set of positive definite matrices of the form
    \begin{align*}
    \Psi = \begin{pmatrix} \delta_1 & 0 & \varrho \\ 0 & \delta_1 & \varrho \\ \varrho & \varrho & \delta_2 \end{pmatrix}.
    \end{align*}
Since the colouring is compatible, the RDAG model is equal to $\Mcal_{A(\Gcal,c)}$ from~\eqref{eqn:RDAG1}. 
Any matrix in $\Mcal_{A(\Gcal,c)}$ satisfies the equalities for the RCON model, so we have containment of the RDAG model in the RCON model. Conversely, given $\Psi$ in the RCON model,
    \begin{align*}
         \det(\Psi) = \delta_1^2 (\delta_2 - 2 \varrho^2 \delta_1^{-1}) > 0 \quad \text{ hence} \quad
         \delta_2 - 2 \varrho^2 \delta_1^{-1}  > 0,
    \end{align*}
since $\Psi$ is positive-definite. 
Positive definiteness of $\Psi$ also implies $\delta_1, \delta_2 > 0$. 
We obtain real numbers $d_1 := \sqrt{\delta_1}$, $r := \varrho / d_1$ and $d_2 := \sqrt{\delta_2 - 2 \varrho^2 \delta_1^{-1}}$, which shows that $\Psi$ is of the form in~\eqref{eqn:RDAG1}.
\end{example}

\begin{example}[RDAG $\neq$ RCON]
Consider the RDAG model on
    \begin{tikzcd}[cramped, sep = small]
        {\color{blue}\circled{1}} & {\color{blue}\circled{2}} \ar[l, red]
    \end{tikzcd},
the graph with two blue (circular) vertices $\{ 1, 2\}$ and a red edge. The colouring is compatible, so by Proposition~\ref{prop:compatibleColouring} the RDAG model is $\Mcal_{A(\Gcal,c)}$, where
$$ A(\Gcal,c) = \left\{ \begin{pmatrix} d & r \\ 0 & d \end{pmatrix} \, \bigg\vert \, d \neq 0 \right\} . $$
The corresponding RCON model is on the coloured undirected graph
    \begin{tikzcd}[cramped, sep = small]
        {\color{blue}\circled{1}} & {\color{blue}\circled{2}} \ar[l, no head, red]
    \end{tikzcd}
. The RCON model consists of all $\Psi \in \PD_2$ with $\Psi_{11} = \Psi_{22}$ and $\Psi_{12} = \Psi_{21}$, by Definition~\ref{def:rcon}.
Neither model is contained in the other: the RCON model contains
    \begin{align*}
        \begin{pmatrix} 4 & 2 \\ 2 & 4 \end{pmatrix} = 
        \begin{pmatrix} 2 & 0 \\ 1 & \sqrt{3} \end{pmatrix} \begin{pmatrix} 2 & 1 \\ 0 & \sqrt{3} \end{pmatrix}
    \end{align*}
but the diagonal entries $2$ and $\sqrt{3}$ in the Cholesky decomposition do not satisfy the conditions for $A(\Gcal,c)$. Conversely, the matrix
    \begin{align*}
        \begin{pmatrix} 1 & 0 \\ 2 & 1 \end{pmatrix} \begin{pmatrix} 1 & 2 \\ 0 & 1 \end{pmatrix}
        = \begin{pmatrix} 1 & 2 \\ 2 & 5 \end{pmatrix}
    \end{align*}
is in the RDAG model, but not the RCON model, since $\Psi_{11} \neq \Psi_{22}$. 
\end{example}

To characterize when an RDAG model is equal to its corresponding RCON model, we give two constructions of coloured graphs obtained from some $(\Gcal,c)$, one that is built from a vertex and the other from an edge. As before, $\Gcal = (I,E)$.

Fix a vertex $i \in I$. Recall that the children of $i$ are the vertices $k$ with $(i \to k) \in E$.
Consider the subgraph on vertex set $\{ i \} \cup {\rm ch}(i)$ with edges $i \to k$ for each $k \in {\rm ch}(i)$, and colours inherited from $(\Gcal, c)$. We denote the subgraph by $\Gcal_i$.

Now fix an edge $(j \to i) \in E$. Consider vertices $\{ i \} \cup \left( {\rm ch}(i)\cap{\rm ch}(j) \right)$ with vertex colours inherited from $(\Gcal,c)$. For each $k \in {\rm ch}(i)\cap{\rm ch}(j)$, we introduce two edges $i \to k$, one with colour $c(ki)$ and the other with colour $c(kj)$. We denote this graph by $\Gcal_{(j \to i)}$.
 
 \begin{example}
 \label{ex:half_jumbled} 
Consider the graph\footnote{\label{note1}with three vertex colours (blue/circular, black/square, and purple/triangular) and four edge colours (red/solid, green/squiggly, orange/dashed, and brown/dotted)}
    \begin{center}
    \begin{tikzcd}[column sep = small, row sep = small,decoration={snake,amplitude=0.8pt}]
        & & & {\color{Fuchsia}\triangled{5}} \ar[dl, orange, dashed] \ar[dll, OliveGreen, bend right = 7,decorate]\ar[dlll, red, bend right = 20] \ar[dd, Maroon,thick,dotted] \\
        \squared{1} & \squared{2} & \squared{3} & \\
        & & & {\color{blue}\circled{4}} \ar[ul, OliveGreen,decorate] \ar[ull, red, bend left = 10]\ar[ulll, orange, bend left = 20, dashed] & 
    \end{tikzcd}
    \end{center}
    The vertex construction at vertex $5$ and edge construction at edge
    $5 \to 4$ are:
    \begin{center} 
    $\Gcal_5 =$\begin{tikzcd}[column sep = small, row sep = small,decoration={snake,amplitude=0.8pt}]
        & & & {\color{Fuchsia}\triangled{5}} \ar[dl, orange,dashed] \ar[dll, OliveGreen, bend right = 10,decorate]\ar[dlll, red, bend right = 20] \ar[d, Maroon,dotted] \\
        \squared{1} & \squared{2} & \squared{3} & {\color{blue}\circled{4}} & 
    \end{tikzcd}
    \qquad \qquad
     $\Gcal_{(5 \to 4)} =$
    \begin{tikzcd}[column sep = scriptsize, row sep = small,decoration={snake,amplitude=0.8pt}]
        \squared{1} & \squared{2} & \squared{3} & {\color{blue}\circled{4}} \ar[l, orange, bend right = 20,dashed] \ar[ll, OliveGreen, bend right = 40,decorate] \ar[lll, red, bend right = 50] \ar[l, OliveGreen, bend left = 30,decorate] \ar[ll, red, bend left = 40]\ar[lll, orange, bend left = 50,dashed]
    \end{tikzcd}
    \end{center}
    \end{example} 
    
Given a DAG $\Gcal$, its corresponding undirected graph is denoted $\Gcal^u$. Similarly, given a coloured DAG $(\Gcal, c)$, its corresponding undirected coloured graph (where the colour of a directed edge becomes the colour of the undirected edge) is $(\Gcal^u,c)$.
Two coloured graphs $(\Gcal,c)$ and $(\Gcal',c')$ are {\em isomorphic} if the coloured graphs are the same up to relabelling vertices. We denote an isomorphism by $\Gcal \simeq \Gcal'$ when the colouring is clear.

 \begin{theorem}\label{thm:RCONequalsRDAG}
Consider the RDAG model on $(\Gcal,c)$ where colouring $c$ is compatible. The 
RDAG model on $(\Gcal,c)$
is equal to the RCON model on $(\Gcal^u,c)$ if and only if:
\begin{itemize}
\item[(a)] $\Gcal$ has no unshielded colliders; 
\item[(b)] $\Gcal_i \simeq \Gcal_j$ for every pair of vertices $i,j$ of the same colour; and
\item[(c)] $\Gcal_{(j \to i)} \simeq \Gcal_{(l \to k)}$ for every pair of edges $j \to i$ and $l \to k$ of the same colour.
\end{itemize}
\end{theorem}

\begin{example}
Our running example  \begin{tikzcd}[cramped, sep = small]
        {\color{blue}\circled{1}} & \squared{3} \ar[r, red] \ar[l, red] & {\color{blue}\circled{2}}
    \end{tikzcd}
    satisfies the 
    conditions of Theorem~\ref{thm:RCONequalsRDAG}: it has no unshielded colliders, the graphs $\Gcal_1$ and $\Gcal_2$ both consist of a single blue vertex, and $\Gcal_{(3\to 1)}$ and $\Gcal_{(3 \to 2)}$ are both isomorphic to \begin{tikzcd}[cramped, sep = small]
        {\color{blue}\circled{1}} & \squared{2}  \ar[l, red]
    \end{tikzcd}.
    The RDAG and RCON models are therefore equivalent, as we saw in Example~\ref{ex:RCONequalsColoredTDAG}. 
    \end{example}

\begin{example}
The following graph also satisfies the conditions of Theorem~\ref{thm:RCONequalsRDAG}:
    \begin{center}
    \begin{tikzcd}[column sep = small, row sep = small,decoration={snake,amplitude=0.8pt}]
        & & & {\color{Fuchsia}\triangled{9}} \ar[dl, orange,dashed] \ar[dll, OliveGreen, bend right = 10,decorate]\ar[dlll, red, bend right = 20] \ar[dd, Maroon,thick,dotted] & {\color{Fuchsia}\tiny\triangled{10}} \ar[dr, orange,dashed] \ar[drr, OliveGreen, bend left = 10,decorate]\ar[drrr, red, bend left = 20] \ar[dd, Maroon,thick,dotted] & & &\\
        \squared{1} & \squared{2} & \squared{3} & & & \squared{4} & \squared{5} & \squared{6} \\
        & & & {\color{blue}\circled{7}} \ar[ul, OliveGreen,decorate] \ar[ull, red, bend left = 10]\ar[ulll, orange, bend left = 20,dashed] & {\color{blue}\circled{8}} \ar[ur, red] \ar[urr, orange, bend right = 10,dashed]\ar[urrr, OliveGreen, bend right = 20,decorate] & & &
    \end{tikzcd}
    \end{center}
\begin{itemize}
    \item[(a)] It has no unshielded colliders.
    \item[(b)] For the black (square) vertices, the graphs $\Gcal_i$ consist of one black vertex. For the blue (circular) vertices, the $\Gcal_i$ are isomorphic to
    \begin{center}
    \begin{tikzcd}[column sep = small, row sep = small,decoration={snake,amplitude=0.8pt}]
        \squared{1} & \squared{2} & \squared{3} & {\color{blue}\circled{4}} \ar[l, orange, bend right = 20,dashed] \ar[ll, OliveGreen, bend right = 30,decorate] \ar[lll, red, bend right = 40]
    \end{tikzcd}
    \end{center}
    The purple (triangular) vertices have $\Gcal_i$ isomorphic to $\Gcal_5$ from Example~\ref{ex:half_jumbled}.
    \item[(c)] All edges $j \to i$ have $\ch(j) \cap \ch(i) = \emptyset$, except for the two brown edges. For these, $\Gcal_{(10 \to 8)}$ and $\Gcal_{(9 \to 7)}$ are both isomorphic to $\Gcal_{(5 \to 4)}$ from Example~\ref{ex:half_jumbled}.
\end{itemize} 
Hence the RDAG model on this coloured graph is equal to the RCON model on the underlying undirected graph.
Note that the two connected components of $(\Gcal,c)$ are not isomorphic. We will see why this is not required for the proof of Theorem~\ref{thm:RCONequalsRDAG}, i.e. why we can collapse vertices $i$ and $j$ in the definition of $\Gcal_{(j \to i)}$.
\end{example}

One ingredient to our proof of Theorem~\ref{thm:RCONequalsRDAG} is the condition for a DAG model to be equal to its corresponding undirected graphical model.

\begin{theorem}[{Gaussian special case of~\cite[Theorem 3.1]{andersson1997markov},~\cite[Theorem 5.6]{frydenberg1990chain}.}]
\label{thm:DAGCONeq}
The DAG model on $\Gcal$ is equal to the undirected Gaussian graphical model on $\Gcal^u$ if and only if $\Gcal$ has no unshielded colliders.
\end{theorem}

We now prove Theorem~\ref{thm:RCONequalsRDAG} via two propositions.

\begin{proposition}\label{prop:RDAGinRCON}
Let $(\Gcal,c)$ be a coloured DAG with compatible colouring $c$. The RDAG model on $(\Gcal, c)$ is contained in the RCON model on $(\Gcal^u, c)$ if and only if
\begin{itemize}
\item[(a)] $\Gcal$ has no unshielded colliders; 
\item[(b)] $\Gcal_i \simeq \Gcal_j$ for every pair of vertices $i,j$ of the same colour; and
\item[(c)] $\Gcal_{(j \to i)} \simeq \Gcal_{(l \to k)}$ for every pair of edges $j \to i$ and $l \to k$ of the same colour. \end{itemize} 
\end{proposition}

\begin{proof}
Since the colouring is compatible, the RDAG model is $\mathcal{M}_{A(\Gcal,c)}$, by Proposition~\ref{prop:compatibleColouring}. Let $a \in A(\Gcal,c)$ be a general matrix. We think of it as having indeterminate entries, one for each vertex colour and one for each edge colour. 

If the RDAG model is contained in the RCON model, we must have $(a\T a)_{ij} = 0$ whenever $a_{ij} = a_{ji} = 0$.  This holds iff there are no unshielded colliders in $\Gcal$, by Theorem~\ref{thm:DAGCONeq}. Moreover, certain equalities must hold on $a\T a$. We have vertex colour condition $(a\T a)_{ii} = (a\T a)_{jj}$ whenever $c(i) = c(j)$ and edge colour condition $(a\T a)_{ij} = (a\T a)_{kl}$ whenever $c(ij) = c(kl)$. These give the polynomial identities
\begin{align}
\label{eqn:vertex_condition} a_{ii}^2 + \sum_{k \in {\rm ch}(i)} a_{ki}^2 & = a_{jj}^2 + \sum_{l \in {\rm ch}(j)} a_{lj}^2  & &\text{whenever} \quad c(i) = c(j) \\
\label{eqn:edge_condition} a_{ii} a_{ij} + \sum_{p \neq i,j}^m a_{pi} a_{pj} & =  a_{kk} a_{kl} + \sum_{q \neq k,l}^m a_{qk} a_{ql} & &\text{whenever} \quad c(ij) \! = \! c(kl). 
\end{align} 
We show that~\eqref{eqn:vertex_condition} is equivalent to (b) and that~\eqref{eqn:edge_condition} is equivalent to (c).

Given two vertices $i,j$ with $c(i) = c(j)$, we have $a_{ii}^2 = a_{jj}^2$. The sums in~\eqref{eqn:vertex_condition} are equal if and only if $\vert \ch(i) \vert = \vert \ch(j) \vert$ and the edge colours in $\Gcal_i$ and $\Gcal_j$ agree (counted with multiplicity). By compatibility, the vertex colours also agree, hence~\eqref{eqn:vertex_condition} is equivalent to $\Gcal_i \simeq \Gcal_j$.

Next, we consider~\eqref{eqn:edge_condition}. 
No terms $a_{ji} a_{jj}$ and $a_{lk} a_{ll}$ appear, since there is no edge $i \to j$ or $k \to l$. The compatibility of the colouring gives $a_{ii} = a_{kk}$. Hence $a_{ii} a_{ij} = a_{kk} a_{kl}$.
A summand in~\eqref{eqn:edge_condition} vanishes unless $p \in \ch(i) \cap \ch(j)$ or $q \in \ch(k) \cap \ch(l)$. The sums are equal if and only if $\vert \ch(i) \cap \ch(j) \vert = \vert \ch(k) \cap \ch(l) \vert$ and the graphs $\Gcal_{(j \to i)}$ and $\Gcal_{(l \to k)}$ are isomorphic on their edge colours. By compatibility of the colouring, the vertex colours of the children also agree and $c(i) = c(k)$, hence~\eqref{eqn:edge_condition} is equivalent to $\Gcal_{(j \to i)} \simeq \Gcal_{(l \to k)}$.
\end{proof}

\begin{proposition}\label{prop:RCONinRDAG}
Let $(\Gcal,c)$ be a coloured DAG with compatible colouring $c$ such that
\begin{itemize}
\item[(a)] $\Gcal$ has no unshielded colliders; 
\item[(b)] $\Gcal_i \simeq \Gcal_j$ for every pair of vertices $i,j$ of the same colour; and
\item[(c)] $\Gcal_{(j \to i)} \simeq \Gcal_{(l \to k)}$ for every pair of edges $j \to i$ and $l \to k$ of the same colour.
\end{itemize}
Then the RCON model on $(\Gcal^u,c)$ is contained in the RDAG model on $(\Gcal,c)$.
\end{proposition}

\begin{proof}
Given a matrix $\Psi$ in the RCON model on $(\Gcal^u, c)$, we show that it is contained in the RDAG model on $(\Gcal, c)$ by showing that the Cholesky decomposition $\Psi = a\T a$ satisfies the conditions of the set $A(\Gcal,c)$ in~\eqref{eqn:Agc}.
The matrix $a$ is upper triangular. Completing it one column at a time shows that its entries are 
    \begin{align}
        a_{l,l} &= \Big( \Psi_{l,l} \, - \sum_{p \in {\rm ch}(l)} a_{p,l}^2 \Big)^{1/2} \label{eq:RCONinRDAGvertex} \\
        a_{i,j} &= \Big( \Psi_{i,j} \, - \sum_{p \in {\rm ch}(i) \cap {\rm ch}(j)} a_{p,i} a_{p,j} \Big) a_{i,i}^{-1}. \label{eq:RCONinRDAGedge}
    \end{align}
    Note that the expression under the square root in~\eqref{eq:RCONinRDAGvertex} is a positive real number, see~\cite[Lecture 23]{trefethen1997numerical}. 
We need to show that $a_{ij} = 0$ for $i \neq j$ with $j \not \to i$ (the support conditions) and that $a_{ij} = a_{kl}$ whenever $c(ij) = c(kl)$ (the colour conditions).

First we show that $a$ satisfies the support conditions of $A(\Gcal,c)$. If there is no edge from $j$ to $i$ in $\Gcal$, then $\Gcal^u$ has no edge between $i$ and $j$, so $\Psi_{i,j} = 0$. Moreover, all products $a_{p,i} a_{p,j}$ vanish, otherwise $i \to p \leftarrow j$ would be an unshielded collider. Hence $a_{ij} = 0$ if $j \not\to i$ in $\Gcal$.

Next, we show that $a$ satisfies the colour conditions of $A(\Gcal,c)$. 
We prove this inductively over the top left $k \times k$ blocks of $a$.
If $k=1$ there are no symmetries to check.
We assume that the top left $k \times k$ submatrix of $a$ satisfies the symmetries. For the induction step, we compare $a_{1,k+1},a_{2,k+1},\ldots,a_{k+1,k+1}$ with each other and with $a_{i,j}$, $i,j \in [k]$.

If there is an edge $(k+1) \to 1$ with same colour as $j \to i$ for $i,j \in [k]$, we
need to show that $a_{1,k+1} = a_{i,j}$.
First, $a_{11} = a_{ii}$ by compatibility and the induction hypothesis, and $\Psi_{i,j} = \Psi_{1,k+1}$ since $\Psi$ is in the RCON model. Moreover, all $a_{pq}$ for $p,q \in [k]$ respect the symmetries. Since $\Gcal_{(j \to i)} \simeq \Gcal_{(k+1 \to 1)}$, the expressions~\eqref{eq:RCONinRDAGedge} for $a_{i,j}$ and $a_{1,k+1}$ are equal.
Proceeding inductively, we show that all entries $a_{2,k+1},\ldots,a_{k,k+1}$ respect the symmetries of $c$.

Finally, if vertex $k+1$ has same colour as vertex $l \in [k]$, we show $a_{k+1,k+1} = a_{l,l}$. We have $\Gcal_l \simeq \Gcal_{k+1}$ by assumption~(b) and $\Psi_{l,l} = \Psi_{k+1,k+1}$, since $\Psi$ is in the RCON model. Furthermore, we have shown that all $a_{p,q}$, where $p \in [k]$ and $q \in [k+1]$, satisfy the colouring $c$. Altogether, we conclude $a_{l,l} = a_{k+1,k+1}$ using \eqref{eq:RCONinRDAGvertex}.
\end{proof}

\begin{proof}[Proof of Theorem~\ref{thm:RCONequalsRDAG}]
If any of conditions (a),(b), and (c) do not hold, this rules out containment of the RDAG model in the RCON model, by Proposition~\ref{prop:RDAGinRCON}, and hence rules out the two models being equal. If conditions (a),(b),(c) hold, we have containment of the RDAG model inside the RCON model (by Proposition~\ref{prop:RDAGinRCON}) and the reverse containment (by Proposition~\ref{prop:RCONinRDAG}).\end{proof}

\section{MLE: existence, uniqueness, and an algorithm}
\label{sec:MLE_existence}

In this section we characterise the existence and uniqueness of the MLE in an RDAG model via linear dependence conditions on certain matrices. 
Our description specialises to give the characterisation of existence and uniqueness of the MLE in a usual DAG model in terms of linear dependence among the rows of the sample matrix.
    
Let $\alpha_s$ be the number of vertices of colour $s$. 
For an edge $j \to i$, vertex $j$ is called the {\em source} of the edge and $i$ is called the {\em target}.
The \emph{parent relationship colours} are the colours of all edges with a target of colour $s$:
    \begin{align*}
        \prc(s) = \lbrace c(ij) \mid \text{there exists } j \to i \text{ in } \mathcal{G} \text{ with } c(i) = s \rbrace, \qquad \beta_s := \vert \prc(s) \vert.
    \end{align*}
A sample matrix is $Y \in \RR^{m \times n}$; its $m$ rows index the vertices in $\Gcal$, and its columns are the $n$ samples.
For each vertex colour in $\Gcal$ we define an augmented sample matrix.

\begin{definition}
\label{def:MYs}
The {\em augmented sample matrix} of sample matrix $Y$ and vertex colour $s$, denoted~$M_{Y,s}$, has size $(\beta_s +1) \times \alpha_s n$. We construct it row by row. 
Each row consists of $\alpha_s$ blocks, each a vector of length $n$.
For notational simplicity, we assume for now that the vertices of colour $s$ are the set $\{ 1,2, \ldots, \alpha_s\} \subset I$. 
Then the top row of $M_{Y,s}$ is 
$$\begin{pmatrix} Y^{(1)} & Y^{(2)} & \ldots & Y^{(\alpha_s)} \end{pmatrix} \in \RR^{\alpha_s n},$$
where $Y^{(i)} \in \RR^n$ is the $i$th row of sample matrix $Y \in 
\RR^{m \times n}$.
The other rows of $M_{Y,s}$ are indexed by the parent relationship colours $\prc(s)$. The row indexed by $t \in \prc(s)$ is
$$
\left( \sum_{\substack{1 \leftarrow j \\ c(1j) = t}} Y^{(j)} \quad \sum_{\substack{2 \leftarrow j \\ c(2j) = t}} Y^{(j)}  \quad \cdots \quad \sum_{\substack{\alpha_s \leftarrow j \\ c(\alpha_s j) = t}} Y^{(j)}  \right) .$$
Note that the sum at the $k$th block is zero if there are no $j \to k$ of colour $t$.
Let $M_{Y,s}^{(i)}$ denote the $i$th row of $M_{Y,s}$, where we index from $0$ to $\beta_s$. 
\end{definition}

\begin{example}
Our running example \begin{tikzcd}[cramped, sep = small]
        {\color{blue}\circled{1}} & \squared{3} \ar[r, red] \ar[l, red] & {\color{blue}\circled{2}}
    \end{tikzcd}
has two augmented sample matrices, one for each vertex colour:
    \begin{equation}\label{eqn:running_ex_MYs} M_{Y, {\color{blue}\circ} } = \begin{pmatrix} Y^{(1)} & Y^{(2)} \\ Y^{(3)} & Y^{(3)} \end{pmatrix} 
    \begin{matrix} {\color{blue}\circ} \\ {\color{red}\to}  \end{matrix} \in \RR^{2 \times 2n} 
    \qquad \text{ and } \qquad
    M_{Y, {\square} } = \begin{pmatrix} Y^{(3)}  \end{pmatrix} \in \RR^{1 \times n} .
    \end{equation} 
\end{example}

\begin{example}\label{ex:augmentedSampleMatrix}
The RDAG model on the coloured DAG \begin{center}
\begin{tikzcd}[column sep = small,decoration={snake,amplitude=1pt}]
        & & {\color{blue}\circled{1}} & &\\
        \squared{3} \ar[rru, red, bend left = 30] \ar[rrd, orange, dashed, bend right = 30] & \squared{4} \ar[ru, Maroon, dotted] \ar[rd, OliveGreen, decorate] & \squared{5} \ar[u, OliveGreen, decorate] \ar[d, orange, dashed] & \squared{6} \ar[lu, Fuchsia, decorate, decoration={zigzag,amplitude=3pt}] \ar[ld, orange, dashed] & \squared{7} \ar[llu, Maroon, dotted, bend right = 30] \\
        & & {\color{blue}\circled{2}} & &
        \end{tikzcd}
        \, has \,
        $M_{Y,{\color{blue}\circ}} =         \begin{pmatrix}
            Y^{(1)} & Y^{(2)} \\ Y^{(3)} & 0 \\ 0 & Y^{(3)} + Y^{(5)} + Y^{(6)} \\ Y^{(5)} & Y^{(4)} \\ Y^{(6)} & 0 \\ Y^{(4)} + Y^{(7)} & 0
        \end{pmatrix}
        \begin{matrix}
            {\color{blue}\circ} \\ \begin{tikzcd}[cramped, sep = small] \ar[r, red] & \  \end{tikzcd} \\ \begin{tikzcd}[cramped, sep=small] \ar[r, orange, dashed] & \  \end{tikzcd} \\ \begin{tikzcd}[cramped, sep=small] \ar[r, OliveGreen, decorate, decoration={snake,amplitude=1.2pt}] & \  \end{tikzcd} \\ \begin{tikzcd}[cramped, sep=small] \ar[r, Fuchsia, decorate, decoration={zigzag,amplitude=3pt}] & \  \end{tikzcd} \\ \begin{tikzcd}[cramped, sep=small] \ar[r, Maroon, dotted] & \  \end{tikzcd}
        \end{matrix}$
\end{center} 
\end{example}

\begin{theorem}\label{thm:MLestimationLinDependence}
Consider the RDAG model on $(\Gcal,c)$ where colouring $c$ is compatible, 
and fix sample matrix $Y \in \RR^{m \times n}$. 
The following possibilities characterise maximum likelihood estimation given $Y$:
\[ \begin{matrix} \text{(a)} & \ell_Y \text{ unbounded from above} & \Leftrightarrow & M_{Y,s}^{(0)} \in \mathrm{span} \big\lbrace M_{Y,s}^{(i)} : i \in [\beta_s] \big\rbrace \text{ for some } s \in c(I) \\ 
\text{(b)} & \text{MLE exists} & \Leftrightarrow &  M_{Y,s}^{(0)} \notin \mathrm{span}  \big\lbrace M_{Y,s}^{(i)} : i \in [\beta_s] \big\rbrace \text{  for all } s \in c(I) \\
\text{(c)} & \text{MLE exists uniquely} &  \Leftrightarrow &  M_{Y,s} \text{ has full row rank for all } s \in c(I). \\ \end{matrix} \]
\end{theorem}

\begin{example}
\label{ex:running5}
For our running example \begin{tikzcd}[cramped, sep = small]
        {\color{blue}\circled{1}} & \squared{3} \ar[r, red] \ar[l, red] & {\color{blue}\circled{2}}
    \end{tikzcd}, Theorem~\ref{thm:MLestimationLinDependence} says that the MLE exists uniquely provided $Y^{(3)} \neq 0$ and $\begin{pmatrix} Y^{(1)} & Y^{(2)} \end{pmatrix}$ is not parallel to $\begin{pmatrix} Y^{(3)} & Y^{(3)} \end{pmatrix}$.
    This holds almost surely as soon as we have one sample, as we mentioned in Example~\ref{ex:very_first}.
\end{example}

\begin{example}
Returning to Example~\ref{ex:augmentedSampleMatrix}, the MLE given $Y$ exists provided $M_{Y, {\square} } = \begin{pmatrix} Y^{(3)} & \cdots & Y^{(7)} \end{pmatrix} \neq 0$, and $\begin{pmatrix} Y^{(1)} & Y^{(2)} \end{pmatrix}$ is not in the span of the other rows of $M_{Y,{\color{blue}\circ}}$. The MLE is unique if and only if $M_{Y,{\color{blue}\circ}}$ is full row rank, since this also implies $M_{Y, {\square} } \neq 0$.
\end{example}

The proof of Theorem~\ref{thm:MLestimationLinDependence} gives Algorithm~\ref{algorithm:the} for finding the MLE in an RDAG model with compatible colouring. The MLE is returned as entries of the matrices $\Lambda$ and $\Omega$. We give the MLE in a closed-form formula, as a collection of least squares estimators.

\begin{remark}
In Algorithm~\ref{algorithm:the}, the entries of $\Omega$ are given as
$\{ \omega_{ss}: s \in c(I) \}$.
The entries of $\Lambda$ are returned as $\{ \lambda_{s,t} : s \in c(I), t \in \prc(s)\}$, which equals the set of edge colours by compatibility, since edge colour $t$ only appears in $\prc(s)$ for one $s$. 
\end{remark}

\begin{algorithm}[h]
\label{algorithm:the} 
\SetAlgoLined
\SetKwInOut{Input}{input} 
\SetKwInOut{Output}{output}
\Input{A coloured DAG $(\Gcal,c)$ and sample matrix $Y$}
\Output{An MLE given $Y$ in the RDAG model on $(\Gcal,c)$, if one exists} 
 \For{$s \in c(I)$}{
 $\alpha_s:= |c^{-1}(s)|$\; 
 $\beta_s:= |\prc(s)|$\;
  construct matrix $M_{Y,s} \in \RR^{(\beta_s + 1) \times \alpha_s n}$\; 
  $P_{Y,s}:=$ orthogonal projection of $M_{Y,s}^{(0)}$ onto ${\rm span} \{ M_{Y,s}^{(t)} : t \in [\beta_s] \}$\;
  \eIf{$P_{Y,s} = M_{Y,s}^{(0)}$}{
   \Return{ MLE does not exist}\;
   }{
   coefficients $\lambda_{s,t}$ are such that $P_{Y,s} = \sum_{t \in \prc(s)} \lambda_{s,t} M_{Y,s}^{(t)}$\;
   $\omega_{ss} := (\alpha_s n)^{-1} \| P_{Y,s} - M_{Y,s}^{(0)}\|^2$\;
  }
 }
 \Return{MLE for $\Lambda$ and $\Omega$}
 \caption{Computing the MLE for an RDAG model}
\end{algorithm}

The proof of Theorem~\ref{thm:MLestimationLinDependence} directly gives a description of the set of MLEs.

\begin{corollary}
\label{cor:the_MLEs}
Consider the RDAG model on $(\Gcal,c)$ where colouring $c$ is compatible, 
with sample matrix $Y \in \RR^{m \times n}$. If $(\Lambda, \Omega)$ and $(\Lambda',\Omega')$ are two MLEs, then $\Omega = \Omega'$ and 
 $$\sum_{t \in \prc(s)} (\lambda_{s,t} - \lambda'_{s,t}) M_{Y,s}^{(t)} = 0, \qquad \text{for all} \quad s \in c(I).$$
\end{corollary}

\subsection{Usual DAG models}

We recall the characterisation of MLE existence and uniqueness for usual Gaussian graphical models, as linear dependence conditions on the sample matrix.
For a DAG $\Gcal$ on $m$ nodes, and $n$ data samples, the sample matrix is $Y \in \RR^{m \times n}$.
For a node $i$ in $\Gcal$ we denote by $Y^{(i)}$ the $i$th row of $Y$, by $Y^{(\pa(i))}$ the sub-matrix of $Y$ with rows indexed by the parents of $i$ in $\Gcal$, and by $Y^{(\pa(i) \cup i)}$ the sub-matrix of $Y$ with rows indexed by node $i$ and its parents. 

\begin{theorem}\label{thm:DAG_Ymatrix}
Consider the DAG model on $\Gcal$, with $m$ nodes, and fix sample matrix $Y \in \RR^{m \times n}$. The following possibilities characterise maximum likelihood estimation given $Y$:
   \[ \begin{matrix} \text{(a)} & \ell_Y \text{ unbounded from above}  & \Leftrightarrow &  Y^{(i)} \in \mathrm{span} \big\lbrace Y^{(j)} : j \in \pa(i)  \big\rbrace \text{ for some } i \in [m] \\ 
\text{(b)} & \text{MLE exists}  & \Leftrightarrow &  Y^{(i)} \notin \mathrm{span} \big\lbrace Y^{(j)} : j \in \pa(i)  \big\rbrace \text{ for all } i \in [m] \\ 
\text{(c)} & \text{MLE exists uniquely} & \Leftrightarrow & Y^{(\pa(i) \cup i)} \text{ has full row rank for all } i \in [m]. \\ \end{matrix} \] 
\end{theorem}

This result follows from viewing maximum likelihood estimation in a DAG as a sequence of regression problems. The acyclicity ensures that the sub-problems are uncoupled. We give a proof, so as to see its generalisation in Theorem~\ref{thm:MLestimationLinDependence}.

\begin{proof}
We denote the entries of the MLEs $\Lambda$ and $\Omega$ by $\hat{\lambda}_{ij}$ and $\hat{\omega}_{kk}$. 
The negative of the log-likelihood $\ell_Y$, in terms of the parameters $\omega_{ii}$ and $\lambda_{ij}$, is
$$ \sum_{i=1}^m \left( \log \omega_{ii} + \frac{1}{n \omega_{ii}} \| Y_i - \sum_{j \in \pa(i)} \lambda_{ij} Y_j \|^2 \right) .$$
We minimise the above expression.
Each parameter only appears in one summand. In the $i$th summand, the $\hat{\lambda}_{ij}$ always exist: they are coefficients of each $Y^{(j)}$ in the orthogonal projection of $Y^{(i)}$ onto the span of $\{ Y^{(j)} : j \in \pa(i)\}$. 
The $\hat{\lambda}_{ij}$ are unique if and only if $Y^{(\pa(i))}$ has full row rank. Let $\zeta_i$ be the residual $\| Y^{(i)} - \sum_{j \in \pa(i)} \hat{\lambda}_{ij} Y^{(j)} \|^2$. Let $\hat{\omega}_{ii}$ minimise $\log(\omega_{ii}) + \frac{\zeta_i}{n \omega_{ii}}$.
If $\zeta_i = 0$, equivalently if $Y^{(i)} \in \mathrm{span} \big\lbrace Y^{(j)} : j \in \pa(i)  \big\rbrace$, then in the limit $\omega_{ii} \to 0$, the log-likelihood tends to infinity and $\hat{\omega}_{ii}$ does not exist.
Otherwise, the minimum is attained uniquely at $\hat{\omega}_{ii} = \zeta_i / n$. Combining these cases gives the theorem.
\end{proof}

\subsection{Proof of Theorem~\ref{thm:MLestimationLinDependence}}

The proof of Theorem~\ref{thm:MLestimationLinDependence} will be similar to the proof for uncoloured models in Theorem~\ref{thm:DAG_Ymatrix}. We start by proving the following lemma.

\begin{lemma}\label{lem:MinimumOfMinusLogLikelihood}
Fix $\alpha > 0$ and, for $\gamma \geq 0$, consider the family of functions
    \[ f_{\gamma} \colon \RR_{>0} \to \RR, \quad x \mapsto \alpha \log(x) + \frac{\gamma}{x}.\]
\begin{itemize}\itemsep 3pt
    \item[(i)] If $\gamma = 0$, then $f_\gamma$ is neither bounded from below nor bounded from above.
        
    \item[(ii)] If $\gamma > 0$, then $f_{\gamma}$ attains a global minimum at $x_0 = \frac{\gamma}{\alpha}$ with function value $f_{\gamma}(\frac{\gamma}{\alpha}) = \alpha(\log(\gamma) - \log(\alpha) + 1)$. 
        
    \item[(iii)] Given $\gamma_1 \geq \gamma_2 > 0$, we have $f_{\gamma_1}(\frac{\gamma_1}{\alpha}) \geq f_{\gamma_2}(\frac{\gamma_2}{\alpha})$ at the global minima.
\end{itemize}
\end{lemma}

\begin{proof}
Part (i) follows from the properties of $\log$. To prove part~(ii), one computes $f_{\gamma}'(x) = \frac{\alpha}{x} - \frac{\gamma}{x^2}$ for $x > 0$. For $x>0$ we have
    \[ f'_{\gamma}(x) = 0  \quad \Leftrightarrow \quad
    \frac{\alpha}{x} = \frac{\gamma}{x^2} \quad \Leftrightarrow \quad
    \alpha x = \gamma \quad \Leftrightarrow \quad
    x = \frac{\gamma}{\alpha}.\]
Thus $x_0 := \frac{\gamma}{\alpha}$ is the only possible local extremum of $f_\gamma$. For $x>0$,
    \[ f'_{\gamma}(x) > 0  \quad \Leftrightarrow \quad
    \frac{\alpha}{x} > \frac{\gamma}{x^2} \quad \Leftrightarrow \quad
    \alpha x > \gamma \quad \Leftrightarrow \quad
    x > \frac{\gamma}{\alpha}.\]
and similarly one has $f'_{\gamma}(x) < 0$ if and only if $x < \frac{\gamma}{\alpha} = x_0$. Therefore, $x_0$ is a global minimum of $f_\gamma$. One directly verifies the function value for $f_{\gamma}(x_0)$, and so part~(iii) follows from the monotonicity of the logarithm.
\end{proof}

\begin{proof}[Proof of Theorem~\ref{thm:MLestimationLinDependence}]
Since colouring $c$ is compatible, the RDAG model equals $\Mcal_{A(\Gcal,c)}$, by Proposition~\ref{prop:compatibleColouring}. That is, for $\Psi = (\id - \Lambda)\T \Omega^{-1} (\id - \Lambda)$ in the RDAG model, the matrix $a = \Omega^{-1/2} (\id - \Lambda)$ is in $\Mcal_{A(\Gcal,c)}$ and satisfies $\Psi = a\T a$. As usual, let $\alpha_s := \vert c^{-1}(s) \vert$ and $\beta_s := \vert \prc(s) \vert$. The entry of $\Omega$ at vertex colour $s$ is denoted $\omega_{ss}$ and the edge colour entries of $\Lambda$ that point towards colour $s$ are labelled by $\lambda_{s,t}$, $t \in [\beta_s]$. Using the construction of the matrices $M_{Y,s} \in \RR^{(\beta_s +1) \times \alpha_s n}$ and that $\det(\id -\Lambda) = 1$, we compute
    \begin{align}
    \begin{split} 
    \label{eqn:finding_MLE} 
        - \ell_Y(\Psi) &= - \log \det(\Psi) + \tr(\Psi S_Y)
        = \log \det(\Omega) + \frac{1}{n} \|a \cdot Y\|^2 \\
        &= \log \bigg(\prod_{s \in c(I)} \omega_{ss}^{\alpha_s} \bigg) + \frac{1}{n} \sum_{s \in c(I)} \| \omega_{ss}^{-1/2} \big( M_{Y,s}^{(0)} - \sum_{t \in [\beta_s]} \lambda_{s,t} M_{Y,s}^{(t)} \big) \|^2 \\
        &= \sum_{s \in c(I)} \alpha_s \log(\omega_{ss}) + \frac{1}{n \omega_{ss}} \| M_{Y,s}^{(0)} - \sum_{t \in [\beta_s]} \lambda_{s,t} M_{Y,s}^{(t)} \|^2.
        \end{split}
    \end{align}
An MLE is a minimiser of the above expression. Each parameter occurs in exactly one of the summands over $s \in c(I)$, because the $\prc(s)$ partition the edge colours by compatibility. We therefore minimise each summand separately. By Lemma~\ref{lem:MinimumOfMinusLogLikelihood}(iii) we can first determine $\hat{\lambda}_{s,t}$, $t \in [\beta_s]$ that minimise
    \begin{equation}\label{eq:RDAGminUnipotentPart}
        \| M_{Y,s}^{(0)} - \sum_{t \in [\beta_s]} \lambda_{s,t} M_{Y,s}^{(t)} \|^2.
    \end{equation}
 Such $\hat{\lambda}_{s,t}$ always exist: they are coefficients in the orthogonal projection $P_{Y,s}$ of $M_{Y,s}^{(0)}$ onto $\mathrm{span} \big\lbrace M_{Y,s}^{(t)} : t \in [\beta_s] \big\rbrace$; i.e.,
    \[ P_{Y,s} = \sum_{t \in [\beta_s]} \hat{\lambda}_{s,t} M_{Y,s}^{(t)}. \]
Furthermore, $\hat{\lambda}_{s,t}$, $t \in [\beta_s]$ are unique if and only if the vectors $M_{Y,s}^{(t)}$, $t \in [\beta_s]$ are linearly independent. Denote the minimum value of \eqref{eq:RDAGminUnipotentPart} by $\zeta_s$. We will apply Lemma~\ref{lem:MinimumOfMinusLogLikelihood} several times with $\gamma_s := \zeta_s/n$.

If $M_{Y,s}^{(0)} \in \mathrm{span} \big\lbrace M_{Y,s}^{(t)} : t \in [\beta_s] \big\rbrace$ for some $s$, then $\zeta_s = 0$ and the summand $\alpha_s \log(\omega_{ss}) + \zeta_s/ (n \omega_{ss})$ is not bounded from below for $\omega_{ss} > 0$, by Lemma~\ref{lem:MinimumOfMinusLogLikelihood}(i). Hence $\ell_Y$ is not bounded from above, e.g. by setting $\omega_{s',s'} = 1$ and $\lambda_{s',t'} = 0$ for all $s' \in c(I)\setminus \{s\}$ and all $t' \in [\beta_{s'}]$. This proves ``$\Leftarrow$'' of (a).

If $M_{Y,s}^{(0)} \notin \mathrm{span} \big\lbrace M_{Y,s}^{(t)} : t \in [\beta_s] \big\rbrace$, equivalently $\zeta_s > 0$, then the summand $\alpha_s \log(\omega_{ss}) + \zeta_s/ (n\omega_{ss})$ has unique minimiser $\hat{\omega}_{ss} = \zeta_s /(n \alpha_s)$, by Lemma~\ref{lem:MinimumOfMinusLogLikelihood}(ii). Hence, an MLE exists if $\zeta_s > 0$ for all $s \in c(I)$, which proves ``$\Leftarrow$'' in (b). As the right-hand sides of (a) and (b) are opposites, we have proved (a) and (b).

Since the $\hat{\omega}_{ss}$ are uniquely determined (if they exist), an MLE (if it exists) is unique if and only if for all $s \in c(I)$ the vectors $M_{Y,s}^{(t)}$, $t \in [\beta_s]$ are linearly independent. In combination with the condition for MLE existence from part (b) we deduce (c).
\end{proof}

\begin{remark}[Comparison with usual DAG models]
\label{rmk:DAGvsRDAG}
The framework of RDAG models includes usual DAG models as a special case; namely, when each colour is used only once. In this case, Theorem~\ref{thm:MLestimationLinDependence} specialises to Theorem~\ref{thm:DAG_Ymatrix}.
We see in the next section that imposing colours in a DAG reduces the threshold number of samples required for existence and uniqueness of the MLE. 
\end{remark}
 
\subsection{Illustrative examples}
\label{sec:illustrative}

In this section we apply RDAG models to some small illustrative examples. We first apply our running example to model the effect of a mother's height on her two daughters' heights.

\begin{example}
\label{ex:heights}
The RDAG model on coloured graph \begin{tikzcd}[cramped, sep = small]
        {\color{blue}\circled{1}} & \squared{3} \ar[r, red] \ar[l, red] & {\color{blue}\circled{2}}
    \end{tikzcd} is given~by \[ y_1 = \lambda y_3 + \epsilon_1 , \quad y_2 = \lambda y_3 + \epsilon_2, \quad y_3 = \epsilon_3, \qquad \epsilon_1, \epsilon_2 \sim N(0,\omega), \quad \epsilon_3 \sim N(0,\omega'). \]
Let
variable $3$ be the height (in cm) of a woman and let variables $1$ and $2$ be, respectively, the heights of her younger and older daughter. Vertices $1$ and $2$ both being blue indicates that, conditional on the mother's height, the variance of the daughter's heights is the same. The edges both being red encodes that the dependence of a daughter's height on the mother's height is the same for both daughters.

 We saw in Example~\ref{ex:running5} that the MLE exists almost surely given one sample.
We use Algorithm~\ref{algorithm:the} to find the MLE,
given one sample where the the younger daughter's height is 159.75cm, the older daughter's height is 161.56, the mother's height is 155.32, and the population mean height is 163.83cm.
Mean-centring the data gives \[ \begin{pmatrix} Y^{(1)} & Y^{(2)} & Y^{(3)} \end{pmatrix} = \begin{pmatrix} 
    -4.08 & -2.27 & -8.51 \end{pmatrix} 
    .\]
    The only black vertex is $3$, and it has no parents, hence $\omega' = \| Y^{(3)} \|^2 = 72.42$. 
    The orthogonal projection of  $\begin{pmatrix} Y^{(1)} & Y^{(2)} \end{pmatrix}$ onto $\mathrm{span} \left\lbrace \begin{pmatrix} Y^{(3)} & Y^{(3)} \end{pmatrix} \right\rbrace$ has coefficient $\lambda = 0.37$ and residual $\omega = \big[ (-3.175+4.08)^2 + (-3.175+2.27)^2 \big]/2 = 0.82$.
    As we would expect, the regression coefficient $\lambda$ is positive and the variance of the daughters' heights conditional on the mother's height is lower than the variance of the mother's height.
\end{example}

We now consider multiple measurements taken in each generation.

\begin{example}
\label{ex:dog}
We consider measurements of the snout length and head length of dogs.
These are the first two of the 
seven morphometric parameters in the study of clinical measurements of dog breeds
in~\cite{momozawa2020genome}. We compare two RDAG models:
\begin{center}
 \begin{tikzcd}[column sep = small,decoration={snake,amplitude=1pt}]
 \squared{1} & {\color{Fuchsia}\triangled{5}}\ar[r,red] \ar[l,red] & \squared{3} \\ 
 {\color{blue}\circled{2}} & {\color{gray}\pentagoned{6}}\ar[r,OliveGreen,decorate]\ar[l,OliveGreen,decorate] & {\color{blue}\circled{4}} 
\end{tikzcd}
\qquad \text{vs.}
\qquad
 \begin{tikzcd}[column sep = small,decoration={snake,amplitude=1pt}]
 \squared{1} & {\color{Fuchsia}\triangled{5}}\ar[r,red] \ar[l,red]\ar[rd,dotted,Maroon]\ar[ld,dotted,Maroon] & \squared{3} \\ 
 {\color{blue}\circled{2}} & {\color{gray}\pentagoned{6}}\ar[r,OliveGreen,decorate]\ar[l,OliveGreen,decorate]\ar[ru,dashed,orange]\ar[lu,dashed,orange] & {\color{blue}\circled{4}} 
\end{tikzcd}
\end{center}
The black/square vertices 1 and 3 are the snout lengths of the two offspring. Blue/circular vertices 2 and 4 are their head lengths. The purple/triangular vertex 5 is the snout length of the parent and grey/pentagonal vertex 6 is the head length of the parent. The edges encode the dependence of the offsprings' traits on those of the parents.

Maximum likelihood estimation in the left hand model is two copies of Example~\ref{ex:heights}, one on the three odd variables, and one on the three even variables. Thus, given one sample a unique MLE exists almost surely. For the right hand model, Theorem~\ref{thm:MLestimationLinDependence} says that an MLE exists provided $Y^{(5)} \neq 0$, $Y^{(6)} \neq 0$  and neither $\begin{pmatrix} Y^{(1)} & Y^{(3)} \end{pmatrix}$ nor $\begin{pmatrix} Y^{(2)} & Y^{(4)} \end{pmatrix}$ are in $\mathrm{span} \left\lbrace \begin{pmatrix} Y^{(5)} & Y^{(5)} \end{pmatrix}, \begin{pmatrix} Y^{(6)} & Y^{(6)} \end{pmatrix} \right\rbrace$.
Hence an MLE exists almost surely with one sample. Moreover, the augmented sample matrices $M_{Y,{\color{blue}\circ}}$ and $M_{Y, {\square} }$ have full row rank almost surely provided $n \geq 2$, hence the MLE exists uniquely with two samples, by Theorem~\ref{thm:MLestimationLinDependence}.
\end{example} 

The vertex colours in an RDAG model could correspond to colours in an experiment, as follows. Fluorescent reporters can be used to take measurements in a cell. In~\cite{shao2021single}, the authors quantify data in a single living bacteria using fluorescent reporters in red, cyan, yellow, and green, see~\cite[Figure 2a]{shao2021single}. Given such measurements taken for a parent cell and its daughter cells, we could consider analogues of Example~\ref{ex:dog} in which, for example, the red fluorescence in a daughter cell only depends on the red fluorescence in the parent, or larger models in which there can also be dependence on the fluorescence of other colours.

\section{Maximum likelihood thresholds}

\label{sec:thresholds}

In the previous section we gave a characterisation of the existence and uniqueness of the MLE based on linear independence conditions. Here we turn this into results that depend only on the coloured graph, and that hold for a generic sample matrix. We give upper and lower bounds on the thresholds for almost sure existence and uniqueness of the MLE. The bounds hold whenever the sample matrix doesn't have extra linear dependencies among its rows. 

For a fixed number of samples, the MLE in an RDAG model may exist but not be unique almost surely, which cannot happen in an uncoloured model. In fact, Example~\ref{ex:ArbitraryLargeGap} gives a family of models for which the gap between the existence and uniqueness thresholds becomes arbitrarily large.

As well as assuming compatibility of the colouring, we often assume in this section that there are no edges between vertices of the same colour. Some of our bounds involve the following notion of generic rank.

\begin{definition}
\label{def:rs} 
Let $M_Y$ be a matrix whose entries are linear combinations of the entries of a matrix $Y$. 
The {\em generic rank} of $M_Y$ is its rank for generic $Y$.
\end{definition}

For $Y \in \RR^{m \times n}$, we often study the generic rank of $M_Y$ by considering it as a symbolic matrix whose entries are linear forms in the $mn$ indeterminates $Y_{ij}$.

\begin{example}\label{ex:ArbitraryLargeGap1}
The graph\footnote{with two vertex colours (blue/circular and black/square) and five edge colours (red/solid, orange/dashed, green/squiggly, purple/zigzag, and brown/dotted)} 
    \begin{center}
        \begin{tikzcd}[column sep = small,decoration={snake,amplitude=1pt}]
        & & {\color{blue}\circled{1}} & &\\
        \squared{3} \ar[rrd, red, bend right = 30] \ar[rru, red, bend left = 30] & \squared{4} \ar[rd, orange,dashed] \ar[ru, orange,dashed] & \squared{5} \ar[d, OliveGreen,decorate] \ar[u, OliveGreen,decorate] & \squared{6} \ar[ld, Fuchsia,decorate,decoration={zigzag,amplitude=3pt}] \ar[lu, Fuchsia,decorate,decoration={zigzag,amplitude=3pt}] & \squared{7} \ar[llu, Maroon, bend right = 30,dotted] \ar[lld, Maroon, bend left = 30,dotted] \\
        & & {\color{blue}\circled{2}} & &
        \end{tikzcd}
        \qquad \text{has} \qquad 
    $   M_{Y,{\color{blue}\circ}} = \begin{pmatrix}
        Y^{(1)} & Y^{(2)} \\ Y^{(3)} & Y^{(3)} \\
        Y^{(4)} & Y^{(4)} \\ Y^{(5)} & Y^{(5)} \\
        Y^{(6)} & Y^{(6)} \\
        Y^{(7)} & Y^{(7)}
        \end{pmatrix}
        \begin{matrix}
            {\color{blue}\circ} \\ \begin{tikzcd}[cramped, sep = small] \ar[r, red] & \  \end{tikzcd} \\ \begin{tikzcd}[cramped, sep=small] \ar[r, orange, dashed] & \  \end{tikzcd} \\ \begin{tikzcd}[cramped, sep=small] \ar[r, OliveGreen, decorate, decoration={snake,amplitude=1.2pt}] & \  \end{tikzcd} \\ \begin{tikzcd}[cramped, sep=small] \ar[r, Fuchsia, decorate, decoration={zigzag,amplitude=3pt}] & \  \end{tikzcd} \\ \begin{tikzcd}[cramped, sep=small] \ar[r, Maroon, dotted] & \  \end{tikzcd}
        \end{matrix}$
    \end{center}
    When $n=1$, the matrix $M_{Y,{\color{blue}\circ}}$ has generic rank two. Removing its top row gives a $5 \times 2$ matrix of generic rank one.
    \end{example} 

Define the augmented sample matrix $M_{Y,s} \in \RR^{(\beta_s + 1) \times \alpha_s n}$ as in Definition~\ref{def:MYs}.
We let $M'_{Y,s} \in \RR^{\beta_s \times \alpha_s n}$ be obtained from $M_{Y,s}$ by removing its top row. 
As before, $\alpha_s$ is the number of vertices of colour $s$ and $\beta_s$ is the number of edge colours of edges to a vertex of colour $s$. Let $r_s$ be the generic rank of $M'_{Y,s}$ when $n=1$. Let $\mlt_e$ (resp. $\mlt_u$) denote the minimal number of samples needed for almost sure existence (resp. uniqueness) of the MLE.

\begin{theorem}\label{thm:RDAGboundsMlt}
Consider the RDAG model on $(\Gcal,c)$ where colouring $c$ is compatible, and $(\Gcal,c)$ has no edges between vertices of the same colour.
We have the following bounds on the thresholds $\mlt_e$ and $\mlt_u$
    \begin{align}
        \label{eq:BoundsMltExistence}
        \max_s \left\lfloor \frac{r_s-1}{\alpha_s-1} \right\rfloor + 1  &\leq \mlt_e \leq \max_s  \left\lfloor \frac{\beta_s}{\alpha_s} \right\rfloor + 1 , \\[5pt]
        \max_s \left\lfloor \frac{\beta_s}{\alpha_s} \right\rfloor + 1 &\leq \mlt_u \leq \max_s \left( \beta_s + 2 - r_s \right). \label{eq:BoundsMltUniqueness}
    \end{align}
\end{theorem}

It remains an open problem to turn the bounds in Theorem~\ref{thm:RDAGboundsMlt} into formulae for $\mlt_e$ and $\mlt_u$. 
\begin{problem}
\label{prob:1}
Determine the maximum likelihood thresholds $\mlt_e$ and $\mlt_u$ of an RDAG model, as formulae involving properties of the graph $\Gcal$ and its colouring $c$.
\end{problem}

Note that the upper bounds for existence and uniqueness are both at most $\max_s \beta_s +1$, which is the threshold for uniqueness in the (uncoloured) DAG case, see~\cite{drton2019maximum}. Hence the RDAG thresholds are always at least as small as the DAG threshold.

\begin{example}\label{ex:ArbitraryLargeGap}
We find the existence and uniqueness thresholds for the RDAG on the graph in Example~\ref{ex:ArbitraryLargeGap1}. 
The black (square) vertices have no children, so the matrix $M_{Y,\square}$ is full rank as soon as $n \geq 1$. The thresholds are therefore determined by the matrix $M_{Y,{\color{blue}\circ}}$. The generic rank of $M'_{Y,{\color{blue}\circ}}$ is one when $n=1$, i.e. $r_{\color{blue}\circ} = 1$. Theorem~\ref{thm:RDAGboundsMlt} and Proposition~\ref{prop:boundMltsUniqueness} then give bounds
    \begin{align*}
        \frac{r_{\color{blue}\circ}-1}{\alpha_{\color{blue}\circ}-1} + 1 = 1 \leq \mlt_e
        \qquad \text{ and } \qquad
        \mlt_u \leq \beta_{\color{blue}\circ} + 1 - r_{\color{blue}\circ} = 5 + 1 - 1 = 5.
    \end{align*}
In fact, both bounds are attained, as follows. When $n=1$, the row $M_{Y,{\color{blue}\circ}}^{(0)} = (Y^{(1)}, Y^{(2)})$ is almost surely not contained in the span of the other rows of $M_{Y,{\color{blue}\circ}}$, hence $\mlt_e = 1$. On the other hand, we need at least $n=5$ samples for generic linear independence of the rows $(Y^{(3)}, Y^{(3)}), \ldots, (Y^{(7)}, Y^{(7)})$.

This example extends to an arbitrary number of vertices, i.e. the graph with $k+2$ vertices, $2$ blue/circular and $k$ black/square, and $2k$ edges of $k$ colours (arranged as in the $k=5$ case above).
Repeating the above argument gives $\mlt_e = 1$ and $\mlt_u = k$.
\end{example}

\begin{remark}
\label{rmk:no_edges_same_colour}
Theorem~\ref{thm:RDAGboundsMlt} applies to all RDAG models with compatible colouring that are equal to some RCON model, because such models never have edges between vertices of the same colour, as follows.
Take some vertex $i$ of minimal index with $i \leftarrow j$ in~$\Gcal$ having $c(i) = c(j)$. Then no children of $i$ have colour $c(i)$, therefore $\Gcal_i \neq \Gcal_j$, a contradiction to Theorem~\ref{thm:RCONequalsRDAG}(b).
\end{remark}

We modify the edges from Examples~\ref{ex:ArbitraryLargeGap1} and~\ref{ex:ArbitraryLargeGap} to see how the thresholds change.

\begin{example}\label{ex:RDAGsThresholds}
Consider the following DAGs, both with compatible colouring
    \begin{center}
        \begin{tikzcd}[column sep = small, decoration={snake,amplitude=1pt}]
        & & {\color{blue}\circled{1}} & &\\
        \squared{3} \ar[rru, red, bend left = 30] & \squared{4} \ar[rd, Maroon, dotted] \ar[ru, orange, dashed] & \squared{5} \ar[u, OliveGreen, decorate] & \squared{6} \ar[lu, Fuchsia, decorate, decoration={zigzag,amplitude=3pt}] & \squared{7} \ar[llu, Maroon, dotted, bend right = 30] \\
        & & {\color{blue}\circled{2}} & &
        \end{tikzcd}
        \qquad \qquad
        \begin{tikzcd}[column sep = small, decoration={snake,amplitude=1pt}]
        & & {\color{blue}\circled{1}} & &\\
        \squared{3} \ar[rru, red, bend left = 30] \ar[rrd, orange, dashed, bend right = 30] & \squared{4} \ar[ru, Maroon, dotted] \ar[rd, OliveGreen, decorate] & \squared{5} \ar[u, OliveGreen, decorate] \ar[d, orange, dashed] & \squared{6} \ar[lu, Fuchsia, decorate, decoration={zigzag,amplitude=3pt}] \ar[ld, orange, dashed] & \squared{7} \ar[llu, Maroon, dotted, bend right = 30] \\
        & & {\color{blue}\circled{2}} & &
        \end{tikzcd}
    \end{center}
Given sample matrix $Y \in \RR^{7 \times n}$, we respectively obtain 
    \begin{equation*}
        M_{Y,{\color{blue}\circ}} = 
        \begin{pmatrix}
            Y^{(1)} & Y^{(2)} \\ Y^{(3)} & 0 \\ Y^{(4)} & 0 \\ Y^{(5)} & 0 \\ Y^{(6)} & 0 \\ Y^{(7)} & Y^{(4)}
        \end{pmatrix}
        \begin{matrix}
            {\color{blue}\circ} \\ \begin{tikzcd}[cramped, sep = small] \ar[r, red] & \  \end{tikzcd} \\ \begin{tikzcd}[cramped, sep=small] \ar[r, orange, dashed] & \  \end{tikzcd} \\ \begin{tikzcd}[cramped, sep=small] \ar[r, OliveGreen, decorate, decoration={snake,amplitude=1.2pt}] & \  \end{tikzcd} \\ \begin{tikzcd}[cramped, sep=small] \ar[r, Fuchsia, decorate, decoration={zigzag,amplitude=3pt}] & \  \end{tikzcd} \\ \begin{tikzcd}[cramped, sep=small] \ar[r, Maroon, dotted] & \  \end{tikzcd}
        \end{matrix}
        \qquad \qquad
        M_{Y,{\color{blue}\circ}} = 
        \begin{pmatrix}
            Y^{(1)} & Y^{(2)} \\ Y^{(3)} & 0 \\ 0 & Y^{(3)} + Y^{(5)} + Y^{(6)} \\ Y^{(5)} & Y^{(4)} \\ Y^{(6)} & 0 \\ Y^{(4)} + Y^{(7)} & 0
        \end{pmatrix}
        \begin{matrix}
            {\color{blue}\circ} \\ \begin{tikzcd}[cramped, sep = small] \ar[r, red] & \  \end{tikzcd} \\ \begin{tikzcd}[cramped, sep=small] \ar[r, orange, dashed] & \  \end{tikzcd} \\ \begin{tikzcd}[cramped, sep=small] \ar[r, OliveGreen, decorate, decoration={snake,amplitude=1.2pt}] & \  \end{tikzcd} \\ \begin{tikzcd}[cramped, sep=small] \ar[r, Fuchsia, decorate, decoration={zigzag,amplitude=3pt}] & \  \end{tikzcd} \\ \begin{tikzcd}[cramped, sep=small] \ar[r, Maroon, dotted] & \  \end{tikzcd}
        \end{matrix}
    \end{equation*}
In both cases we have $\alpha_{\color{blue}\circ} = 2$, $\beta_{\color{blue}\circ} = 5$, and $r_{\color{blue}\circ} = 2$. Thus, Theorem~\ref{thm:RDAGboundsMlt} gives the bounds
    \begin{align*}
        2 = \left\lfloor \frac{r_{\color{blue}\circ}-1}{\alpha_{\color{blue}\circ}-1} \right\rfloor + 1 \leq \mlt_e \leq \left\lfloor \frac{\beta_{\color{blue}\circ}}{\alpha_{\color{blue}\circ}} \right\rfloor + 1 = 3
        \\
        3 = \left\lfloor \frac{\beta_{\color{blue}\circ}}{\alpha_{\color{blue}\circ}} \right\rfloor + 1 \leq \mlt_u \leq \beta_{\color{blue}\circ} + 2 - r_{\color{blue}\circ} = 5.
    \end{align*}
In fact, Proposition~\ref{prop:boundMltsUniqueness} gives an upper bound of $4$ on $\mlt_u$, since $r_{\color{blue}\circ} \neq \beta_{\color{blue}\circ} + 1 - (\beta_{\color{blue}\circ} / \alpha_{\color{blue}\circ})$.

First, we study the left-hand RDAG. When $n=2$ the row $Y^{(2)} \in \RR^{1 \times 2}$ is generically not in the span of $Y^{(4)}$, hence $M_{Y,{\color{blue}\circ}}^{(0)} = (Y^{(1)}, Y^{(2)})$ is not in the linear span of the other five rows of $M_{Y,{\color{blue}\circ}}$, and we deduce $\mlt_e = 2$. For generic uniqueness of an MLE we need $M_{Y,{\color{blue}\circ}} \in \RR^{6 \times 2n}$ to have full row rank six. For $n \geq 2$, the submatrix
    \begin{align*}
        \begin{pmatrix}
            Y^{(2)} \\ Y^{(4)}
        \end{pmatrix} \in \RR^{2 \times n}
    \end{align*}
generically has rank two. Therefore, $M_{Y,{\color{blue}\circ}}$ has rank at most five if $n = 3$, but $n=4$ suffices for $M_{Y,{\color{blue}\circ}}$ generically having full row rank. We conclude $\mlt_u = 4$.

Next, we study the right-hand RDAG. For $n=2$, $M_{Y,{\color{blue}\circ}}^{(0)} = (Y^{(1)}, Y^{(2)})$ is generically contained in the linear span of the other rows of $M_{Y,{\color{blue}\circ}}$. From the bounds we conclude that $\mlt_e = 3$. For uniqueness, when $n = 3$ the submatrices
    \begin{align*}
        \begin{pmatrix}
            Y^{(3)} \\ Y^{(6)} \\ Y^{(4)} + Y^{(7)}
        \end{pmatrix},
        \quad
        \begin{pmatrix}
            Y^{(2)} \\ Y^{(3)} + Y^{(5)} + Y^{(6)} \\ Y^{(4)}
        \end{pmatrix} \in \RR^{3 \times 3}
    \end{align*}
of $M_{Y,{\color{blue}\circ}}$ generically have rank three, and the zero pattern then ensures that $M_{Y,{\color{blue}\circ}}$ has full row rank six. Combining with the lower bound $3 \leq \mlt_u$ gives $\mlt_u = 3$.
\end{example}

\subsection{Proof of Theorem~\ref{thm:RDAGboundsMlt}}

For fixed vertex colour $s$, we define $\mlt_e(s)$ to be the smallest $n$ such that the top row of $M_{Y,s}$ is almost surely not in the span of the other $\beta_s$ rows, and define $\mlt_u(s)$ to be the smallest $n$ such that the matrix $M_{Y,s}$ is almost surely of full row rank $\beta_s + 1$. 
To prove Theorem~\ref{thm:RDAGboundsMlt}, we use the following lemma.

\begin{lemma}\label{lem:HelpMLTsUniqueness}
Consider the RDAG model on $(\Gcal,c)$ where colouring $c$ is compatible, and fix a vertex colour~$s$. For $n \geq \beta_s$ and generic $Y \in \RR^{m \times n}$ the row vectors  $M_{Y,s}^{(1)}, \ldots, M_{Y,s}^{(\beta_s)}$ are linearly independent.
\end{lemma}

\begin{proof}
We think of the $mn$ entries of $Y$ as indeterminates. Let $M_{Y,s}' \in \RR^{\beta_s \times \alpha_s n}$ be the matrix with rows $M_{Y,s}^{(1)}, \ldots, M_{Y,s}^{(\beta_s)}$. We construct an invertible $\beta_s \times \beta_s$ submatrix of $M_{Y,s}'$.

Without loss of generality, let $1,2,\ldots,\alpha_s$ be the vertices of colour $s$. The matrix $M_{Y,s}'$ has $\alpha_s$ many $\beta_s \times n$ blocks. For each parent relationship colour $p_t$, $t \in [\beta_s]$ there is some vertex $i = i(t) \in [\alpha_s]$ such that there is an edge of colour $p_t$ pointing towards vertex $i = i(t)$. That is, the $i^{th}$ block of $M_{Y,s}'$ has non-zero entries in the $t^{th}$ row. Let $C_t$ be the $t^{th}$ column of that block, which exists as $n \geq \beta_s$. By construction, the $t^{th}$ entry of $C_t$ is non-zero. We show that the matrix $C = \begin{pmatrix} C_{1} & C_2 & \ldots & C_{\beta_s}\end{pmatrix}$, is invertible.

An entry of $C$ is either a sum of variables or it is zero. By construction, column $C_t$ only contains (sums of) elements of the $t^{th}$ column of $Y$. The same variable $Y_{j,t}$ cannot occur in two different entries of $C_t$, because there is at most one edge from $j$ to vertex $i(t)$. Altogether, the entries of $C$ are (possible empty) sums of variables and each variable occurs in at most one entry of $C$. The determinant is an alternating sum over products of permutations, and it is enough to show that one product is non-zero. By construction,  $C_{11} C_{22} \cdots C_{\beta_s \beta_s} \neq 0$. Thus, $M_{Y,s}'$ has rank $\beta_s$ for $n \geq \beta_s$.
\end{proof}

\begin{proposition}\label{prop:boundMltsExistence}
Consider the RDAG model on $(\Gcal,c)$ where colouring $c$ is compatible,
with no edges between any vertices of colour $s$.
If $\alpha_s = 1$, then $\mlt_e(s) = \beta_s + 1$, while
if $\alpha_s \geq 2$ we have
$$ \left\lfloor \frac{r_s-1}{\alpha_s-1} \right\rfloor + 1  \leq \mlt_e(s) \leq \left\lfloor \frac{\beta_s}{\alpha_s} \right\rfloor + 1 .$$
\end{proposition}

\begin{proof}
If $\alpha_s = 1$, the equality $\mlt_e(s) = \beta_s + 1$ is known from the uncoloured case. It remains to consider $\alpha_s \geq 2$.
For the upper bound, we show that if $n > \frac{\beta_s}{\alpha_s}$, then the top row of $M_{Y,s}$ is generically not in the span of the other rows. Since there are no edges between two vertices of colour $s$, the $n \alpha_s$ entries of the top row $M_{Y,s}^{(0)}$ are all independent, from each other and from the entries of the other rows. If $\beta_s < \alpha_s n$, the other $\beta_s$ rows do not span $\RR^{n \alpha_s}$, so a generic choice of top row will not lie in their span. 

For the lower bound, consider the $\beta_s \times \alpha_s$ matrix $M_{Y,s}'$ for $n=1$. Its generic rank is denoted $r_s$. We consider the $1 \times n$ matrix blow up, where the scalar variables $Y^{(i)}$ are replaced by generic row vectors of length $n$, to give a $\beta_s \times \alpha_s n$ matrix. We consider the rank of this matrix as $n$ increases. The maximum rank is $\beta_s$, which occurs for $n \geq \beta_s$ by Lemma~\ref{lem:HelpMLTsUniqueness}. Moreover, the rank is a (weakly) concave function in $n$ \cite[Corollary~2.8]{derksen2017polynomial}. Since the rank is positive integer valued, it cannot be the same at two distinct $n$ unless it is at its maximum. Hence the rank for fixed $n$ is at least $\min ( r_s + n-1, \beta_s)$. Therefore, the top row is in the span of the others whenever 
$$
\min(r_s + n -1, \beta_s) \geq n \alpha_s , \quad \text{in particular, whenever} \quad \alpha_s n \leq r_s + n - 1 \leq \beta_s ,$$
i.e. for 
$n \leq \min \left( \left\lfloor\frac{r_s-1}{\alpha_s-1}\right\rfloor , \beta_s + 1 - r_s \right)$, so $\mlt_e(s) \geq \min \left( \left\lfloor\frac{r_s-1}{\alpha_s-1}\right\rfloor + 1 , \beta_s + 2 - r_s \right)$.
We exclude the possibility that the smaller of the two arguments in the minimum is $\beta_s + 2 - r_s$ by appealing to Proposition~\ref{prop:boundMltsUniqueness}. 
\end{proof}

\begin{proposition}\label{prop:boundMltsUniqueness}
Consider the RDAG model on $(\Gcal,c)$ where colouring $c$ is compatible, 
with no edges between any vertices of colour $s$. Then
    \[\left\lfloor \frac{\beta_s}{\alpha_s} \right\rfloor + 1 \leq \mlt_u(s) \leq \beta_s + 2 - r_s. \]
Moreover, if $r_s \neq \beta_s + 1 - (\beta_s / \alpha_s)$ then $\mlt_u(s) \leq \beta_s + 1 - r_s$.
\end{proposition}

\begin{proof}
For the lower bound, we observe that if $\alpha_s n \leq \beta_s$, the $\beta_s + 1$ rows of $M_{Y,s}$ will be linearly dependent. Hence, we need $n > \frac{\beta_s}{\alpha_s}$.

For the upper bound, let $M'_{Y,s}$ and $r_s$ be as above. Recall from the proof of Proposition~\ref{prop:boundMltsExistence} that, for $n$ samples, $\rk(M_{Y,s}') \geq \min(r_s + n - 1, \beta_s)$ generically.
Thus, for $n = \beta_s + 1 - r_s$ the matrix $M'_{Y,s}$ generically has full row rank $\beta_s$.
It remains to consider the top row of $M_{Y,s}$. We must have $\beta_s \leq \alpha_s n$, otherwise the $\beta_s \times \alpha_s n$ matrix could not have full row rank. We look separately at the possible cases: $\beta_s < n \alpha_s$ and $\beta_s = n \alpha_s$.
If $\beta_s < n \alpha_s$, the row vector $M_{Y,s}^{(0)} \in \RR^{n \alpha_s}$ is generically not in the span of the $\beta_s$ rows of $M'_{Y,s}$, because there are no edges between vertices of colour $s$. Thus, $M_{Y,s}$ generically has full row rank $\beta_s + 1$, and $\mlt_u(s) \leq n = \beta_s + 1 - r_s$. If $\beta_s = n \alpha_s$, equivalently if $r_s = \beta_s + 1 - (\beta_s / \alpha_s)$, an additional sample ensures $\rk(M_{Y,s}) = \beta_s + 1$ generically.
\end{proof}

\begin{proof}[Proof of Theorem~\ref{thm:RDAGboundsMlt}]
We have $\mlt_e = \max_s \mlt_e(s)$ and $\mlt_u = \max_s \mlt_u(s)$ by Theorem~\ref{thm:MLestimationLinDependence} parts~(b) and (c).
Taking the maximum of the lower and upper bounds in Propositions~\ref{prop:boundMltsExistence} and~\ref{prop:boundMltsUniqueness}, over all vertex colours, gives the stated bounds.
\end{proof} 

\subsection{A randomised algorithm}

\begin{proposition}\label{prop:RDAGcomputingMlt}
For an RDAG model on $(\Gcal,c)$, where colouring $c$ is compatible, there is a randomised algorithm for computing the thresholds $\mlt_e$ and $\mlt_u$.
\end{proposition}

\begin{proof}
It suffices to give a randomised algorithm to compute $\mlt_e(s)$ and $\mlt_u(s)$ for fixed vertex colour $s$.
The rank of a symbolic matrix can be computed (efficiently) by a randomised algorithm, see e.g. \cite{Lovasz, Schwartz}. Hence, thinking of the entries of $Y \in \RR^{m \times n}$ as indeterminates, we can compute for any $n \geq 1$ the rank of the symbolic $(\beta_s + 1) \times \alpha_s n$ matrix $M_{Y,s}$ as well as the rank of the symbolic $\beta_s \times \alpha_s n$ matrix $M'_{Y,s}$. We obtain $\mlt_e(s)$ as the smallest $n$ such that $\rk (M_{Y,s}) > \rk(M'_{Y,s})$ and $\mlt_u(s)$ as the smallest $n$ such that $\rk (M_{Y,s}) = \beta_s + 1$. The algorithm terminates by the upper bound of $\beta_s + 1$ for both $\mlt_e(s)$ and $\mlt_u(s)$.
\end{proof}

\section{Simulations}
\label{sec:simulations}

In the previous section, we gave upper and lower bounds for the maximum likelihood thresholds for RDAG models, see Theorem~\ref{thm:RDAGboundsMlt}. Our bounds quantify how the graph colouring serves to decrease the number of samples needed for existence and uniqueness of the MLE. In this section, we 
assume that the number of samples is above the maximum likelihood threshold. We
explore via simulations the distance of an MLE to the true model parameters.
We compare the RDAG model estimate 
from Algorithm~\ref{algorithm:the} to the usual (uncoloured) DAG model MLE.

The details of our simulations are as follows.
We used the NetworkX Python package~\cite{hagberg2008exploring} to build an RDAG model via the following steps.
We first build a DAG by generating an undirected graph according to an Erdős–Rényi model that includes each edge with fixed probability, and then directing the edges so that $j \to i$ implies $j > i$. We assign edge colours randomly, after fixing the total number of possible edge colours. We choose the unique vertex colouring with the largest number of vertex colours that satisfies the compatibility assumption from Definition~\ref{dfn:compatibleColoring}. We sample edge weights $\lambda_{st}$ from a uniform distribution on $[-1, -0.25] \cup [0.25, 1]$ and we sample noise variances $\omega_{ss}$ uniformly from $[0,1]$. Our code is available at \url{https://github.com/seigal/rdag}.

The RDAG MLE is generally closer to the true model parameters than the DAG MLE, see Figure~\ref{fig:1}. As we would expect, both estimates get closer to the true parameters as the number of samples from the distribution increases. At a high number of samples, the difference between the RDAG MLE and the DAG MLE is smaller than at a low number of samples.

\begin{figure}[ht]
    \centering
    \includegraphics[width=8cm]{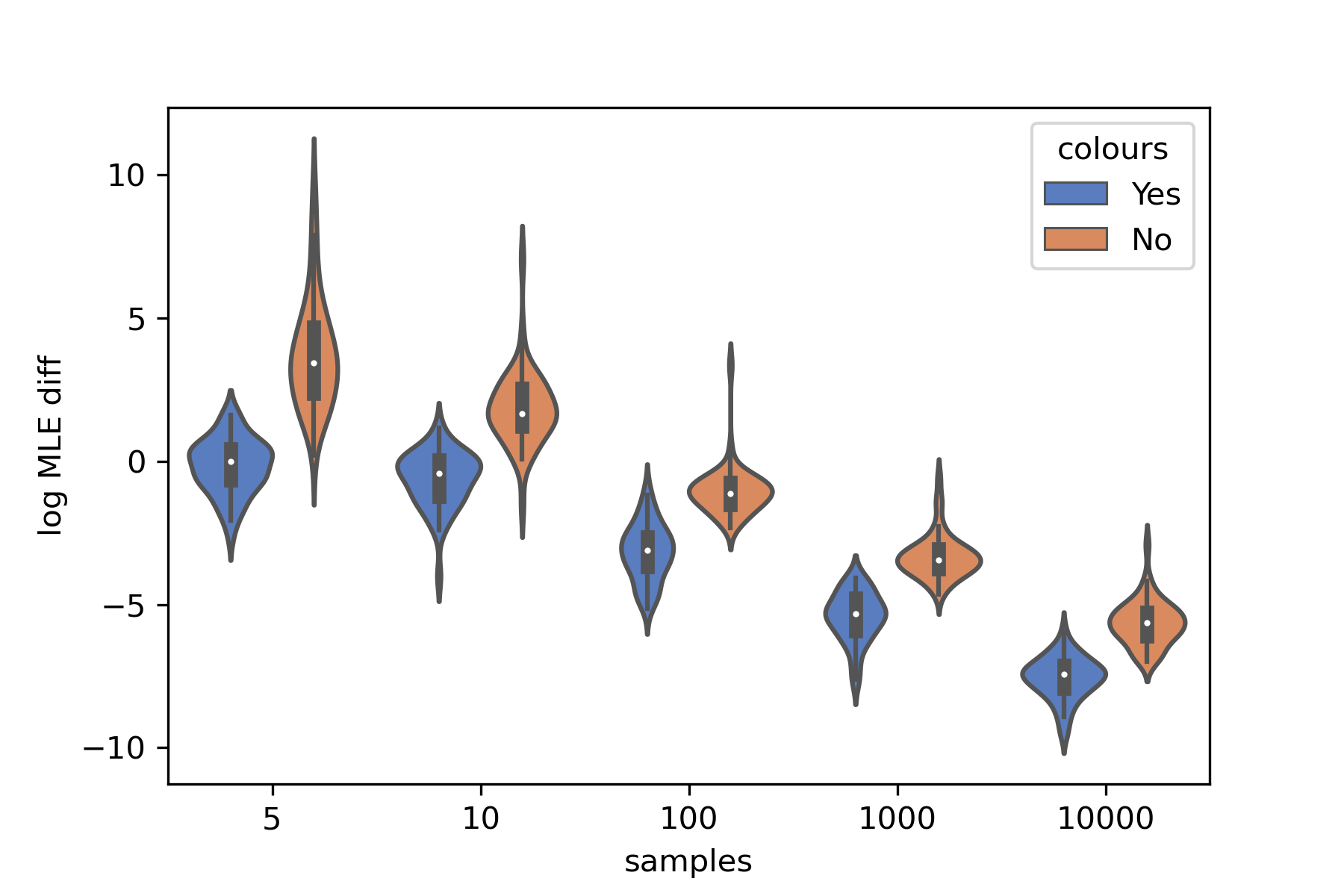}
    \caption{We generated RDAGs on $10$ vertices, with each edge present with probability $0.5$ and $5$ edge colours. We sampled from the distribution $n \in \{ 5, 10, 100, 1000, 10000 \}$ times. For each $n$ we generated $50$ random graphs and computed the RDAG MLE and the DAG MLE, comparing them to the true parameter values on a log scale. The results are displayed in a violin plot, with blue for the RDAG MLE and orange for the DAG MLE.}
    \label{fig:1}
\end{figure}

Next we examined how the RDAG MLE was affected by the number of edge colours, see Figure~\ref{fig:2}. The RDAG MLE is closest to the true parameters when the number of edge colours is small; i.e., when there are fewer parameters to learn. As the number of edge colours increases, the difference between the RDAG MLE and the DAG MLE decreases. Note that the DAG model is the setting where each vertex and edge has its own colour.

\begin{figure}[ht]
    \centering
    \includegraphics[width=8cm]{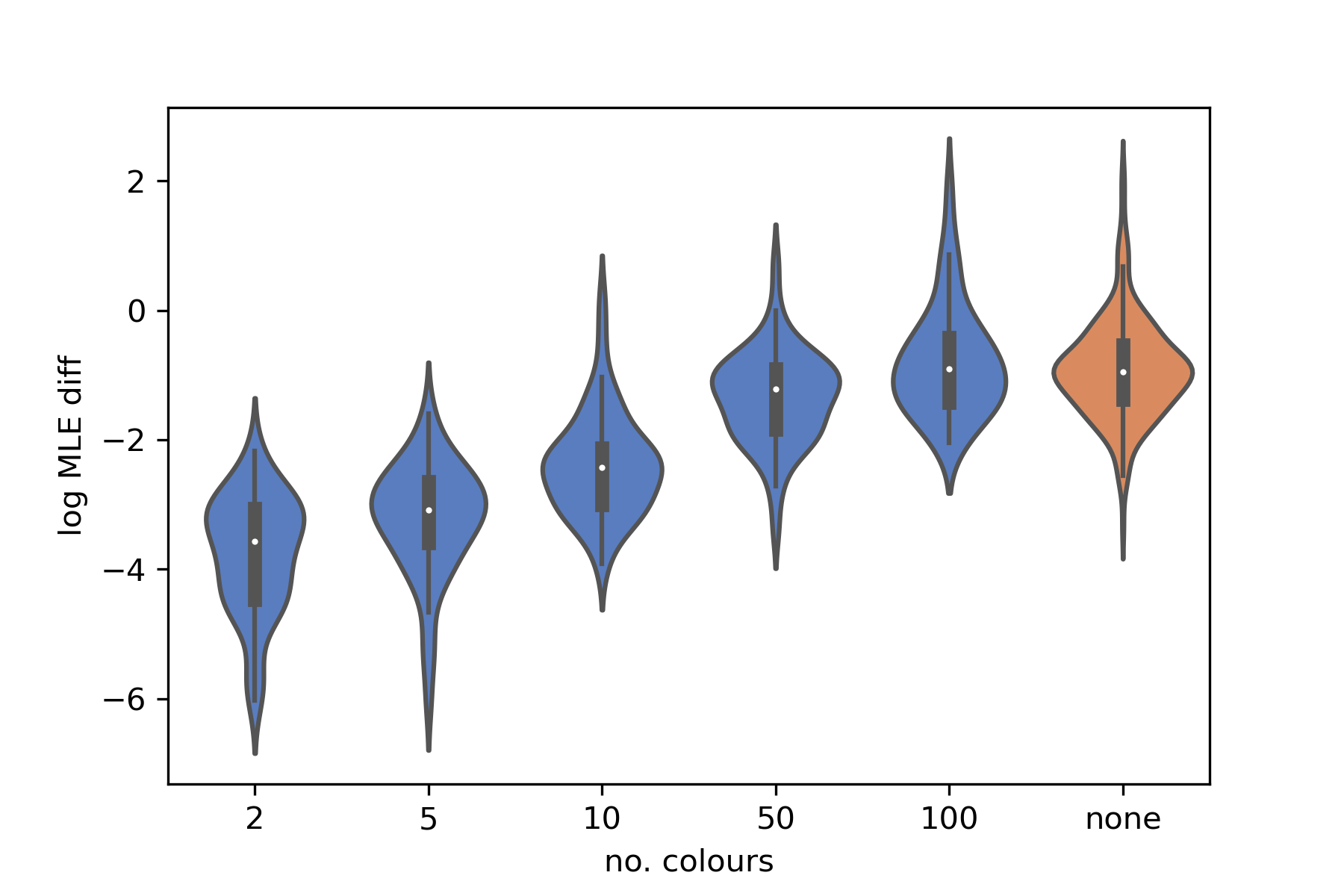}
    \caption{We generated RDAGs on $10$ vertices, each edge present with probability $0.5$ and number of edge colours in $\{ 2, 5, 10, 50, 100 \}$. We sampled from the distribution $100$ times and compared the MLE to the true parameter values on a log scale. The DAG MLE is shown in orange for comparison.}
    \label{fig:2}
\end{figure}

Finally, we looked at how the RDAG MLE and DAG MLE are affected by the edge density of the graph, see Figure~\ref{fig:3}. The RDAG MLEs get closer to the true parameter values as the edge density increases: more edges have the same weight, so more samples contribute to estimating each edge weight. By comparison, the DAG MLEs get further from the true parameters as the edge density increases, because there are more parameters to learn.

\begin{figure}[ht]
    \centering
    \includegraphics[width=8cm]{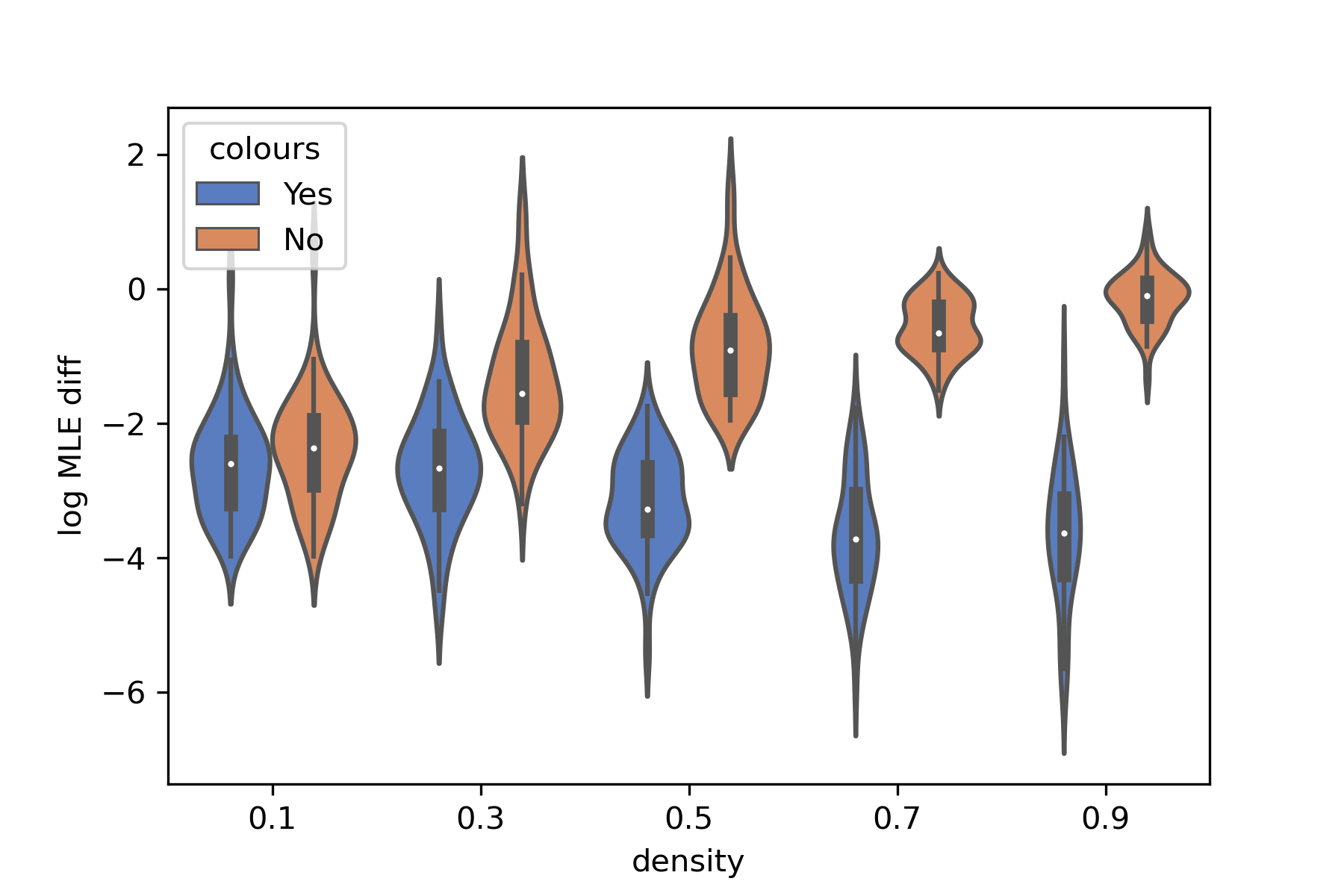}
    \caption{We generated RDAGs on $10$ vertices, each edge present with probability in $\{ 0.1, 0.3, 0.5, 0.7, 0.9\}$ and $5$ edge colours. For each edge probability we generated $50$ random graphs, sampled from each one $100$ times, and compared the RDAG and DAG MLEs. As above, blue is the RDAG MLE and orange is the DAG MLE.}
    \label{fig:3}
\end{figure}

\medskip
{\small
\paragraph{\textbf{Acknowledgements}}
We are grateful to Dominic Bunnett, Mathias Drton, Robin Evans, Caroline Uhler, and Piotr Zwiernik for helpful discussions. VM was supported by the University of Melbourne and NSF Grant CCF 1900460. PR was funded by the European Research Council (ERC) under the European’s Horizon 2020 research and innovation programme (grant agreement no. 787840).
}

\addtocontents{toc}{\protect\setcounter{tocdepth}{0}}

\bibliographystyle{alpha}
 \bibliography{literatur}

\newcommand{\etalchar}[1]{$^{#1}$}
\begin{thebibliography}{MDLW18}

\bibitem[AFS16]{abbruzzo2016operational}
Antonino Abbruzzo, Vincenzo Fasone, and Raffaele Scuderi.
\newblock Operational and financial performance of italian airport companies: A
  dynamic graphical model.
\newblock {\em Transport Policy}, 52:231--237, 2016.

\bibitem[AKRS21]{amendola2020invariant}
Carlos Am{\'e}ndola, Kathl{\'e}n Kohn, Philipp Reichenbach, and Anna Seigal.
\newblock Invariant theory and scaling algorithms for maximum likelihood
  estimation.
\newblock {\em SIAM Journal on Applied Algebra and Geometry}, 5(2):304--337,
  2021.

\bibitem[AM98]{andersson1998symmetry}
Steen Andersson and Jesper Madsen.
\newblock Symmetry and lattice conditional independence in a multivariate
  normal distribution.
\newblock {\em The Annals of Statistics}, 26(2):525--572, 1998.

\bibitem[AMP97]{andersson1997markov}
Steen Andersson, David Madigan, and Michael Perlman.
\newblock On the {M}arkov equivalence of chain graphs, undirected graphs, and
  acyclic digraphs.
\newblock {\em Scandinavian Journal of Statistics}, 24(1):81--102, 1997.

\bibitem[Buh93]{buhl1993existence}
S{\o}ren~L Buhl.
\newblock On the existence of maximum likelihood estimators for graphical
  {G}aussian models.
\newblock {\em Scandinavian Journal of Statistics}, pages 263--270, 1993.

\bibitem[Dem72]{dempster1972covariance}
Arthur~P Dempster.
\newblock Covariance selection.
\newblock {\em Biometrics}, pages 157--175, 1972.

\bibitem[DFKP19]{drton2019maximum}
Mathias Drton, Christopher Fox, Andreas K{\"a}ufl, and Guillaume Pouliot.
\newblock The maximum likelihood threshold of a path diagram.
\newblock {\em The Annals of Statistics}, 47(3):1536--1553, 2019.

\bibitem[DM17]{derksen2017polynomial}
Harm Derksen and Visu Makam.
\newblock Polynomial degree bounds for matrix semi-invariants.
\newblock {\em Advances in Mathematics}, 310:44--63, 2017.

\bibitem[DM21]{derksen2021maximum}
Harm Derksen and Visu Makam.
\newblock Maximum likelihood estimation for matrix normal models via quiver
  representations.
\newblock {\em SIAM Journal on Applied Algebra and Geometry}, 5(2):338--365,
  2021.

\bibitem[DMW20]{derksen2020maximum}
Harm Derksen, Visu Makam, and Michael Walter.
\newblock Maximum likelihood estimation for tensor normal models via castling
  transforms.
\newblock {\em arXiv preprint arXiv:2011.03849}, 2020.

\bibitem[Dol03]{Dolgachev}
Igor Dolgachev.
\newblock {\em Lectures on invariant theory}, volume 296 of {\em London
  Mathematical Society Lecture Note Series}.
\newblock Cambridge University Press, Cambridge, 2003.

\bibitem[DWW14]{danaher2014joint}
Patrick Danaher, Pei Wang, and Daniela~M Witten.
\newblock The joint graphical lasso for inverse covariance estimation across
  multiple classes.
\newblock {\em Journal of the Royal Statistical Society. Series B, Statistical
  methodology}, 76(2):373, 2014.

\bibitem[FHT08]{friedman2008sparse}
Jerome Friedman, Trevor Hastie, and Robert Tibshirani.
\newblock Sparse inverse covariance estimation with the graphical lasso.
\newblock {\em Biostatistics}, 9(3):432--441, 2008.

\bibitem[FLNP00]{friedman2000using}
Nir Friedman, Michal Linial, Iftach Nachman, and Dana Pe'er.
\newblock Using {B}ayesian networks to analyze expression data.
\newblock {\em Journal of computational biology}, 7(3-4):601--620, 2000.

\bibitem[Fou12]{foulds2012graph}
Leslie~R Foulds.
\newblock {\em Graph theory applications}.
\newblock Springer Science \& Business Media, 2012.

\bibitem[Fry90]{frydenberg1990chain}
Morten Frydenberg.
\newblock The chain graph {M}arkov property.
\newblock {\em Scandinavian Journal of Statistics}, pages 333--353, 1990.

\bibitem[GM15]{gao2015estimation}
Xin Gao and H{\'e}l{\`e}ne Massam.
\newblock Estimation of symmetry-constrained {G}aussian graphical models:
  application to clustered dense networks.
\newblock {\em Journal of Computational and Graphical Statistics},
  24(4):909--929, 2015.

\bibitem[GS18]{gross2018maximum}
Elizabeth Gross and Seth Sullivant.
\newblock The maximum likelihood threshold of a graph.
\newblock {\em Bernoulli}, 24(1):386--407, 2018.

\bibitem[HL08]{hojsgaard2008graphical}
S{\o}ren H{\o}jsgaard and Steffen~L Lauritzen.
\newblock Graphical {G}aussian models with edge and vertex symmetries.
\newblock {\em Journal of the Royal Statistical Society: Series B (Statistical
  Methodology)}, 70(5):1005--1027, 2008.

\bibitem[HSSC08]{hagberg2008exploring}
Aric Hagberg, Pieter Swart, and Daniel S~Chult.
\newblock Exploring network structure, dynamics, and function using networkx.
\newblock Technical report, Los Alamos National Lab.(LANL), Los Alamos, NM
  (United States), 2008.

\bibitem[Lau96]{lauritzen1996graphical}
Steffen~L Lauritzen.
\newblock {\em Graphical models}, volume~17.
\newblock Clarendon Press, 1996.

\bibitem[Lov79]{Lovasz}
L.~Lov\'{a}sz.
\newblock On determinants, matchings, and random algorithms.
\newblock In {\em Fundamentals of computation theory ({P}roc. {C}onf.
  {A}lgebraic, {A}rith. and {C}ategorical {M}ethods in {C}omput. {T}heory)},
  volume~2 of {\em Math. Res.}, pages 565--574. Akademie-Verlag, Berlin, 1979.

\bibitem[Mad00]{madsen2000invariant}
Jesper Madsen.
\newblock Invariant normal models with recursive graphical {M}arkov structure.
\newblock {\em Annals of statistics}, pages 1150--1178, 2000.

\bibitem[MDLW18]{maathuis2018handbook}
Marloes Maathuis, Mathias Drton, Steffen Lauritzen, and Martin Wainwright.
\newblock {\em Handbook of graphical models}.
\newblock CRC Press, 2018.

\bibitem[MFK94]{MumfordGIT}
D.~Mumford, J.~Fogarty, and F.~Kirwan.
\newblock {\em Geometric invariant theory}, volume~34 of {\em Ergebnisse der
  Mathematik und ihrer Grenzgebiete (2) [Results in Mathematics and Related
  Areas (2)]}.
\newblock Springer-Verlag, Berlin, third edition, 1994.

\bibitem[MMB{\etalchar{+}}20]{momozawa2020genome}
Yukihide Momozawa, Anne-Christine Merveille, G{\'e}raldine Battaille, Maria
  Wiberg, J{\o}rgen Koch, Jakob~Lundgren Willesen, Helle~Friis Proschowsky,
  Vassiliki Gouni, Val{\'e}rie Chetboul, Laurent Tiret, et~al.
\newblock Genome wide association study of 40 clinical measurements in eight
  dog breeds.
\newblock {\em Scientific reports}, 10(1):1--11, 2020.

\bibitem[PB14]{peters2014identifiability}
Jonas Peters and Peter B{\"u}hlmann.
\newblock Identifiability of gaussian structural equation models with equal
  error variances.
\newblock {\em Biometrika}, 101(1):219--228, 2014.

\bibitem[Pea09]{pearl2009causality}
Judea Pearl.
\newblock {\em Causality}.
\newblock Cambridge university press, 2009.

\bibitem[Pop89]{popov1989closed}
Vladimir~Leonidovich Popov.
\newblock Closed orbits of {B}orel subgroups.
\newblock {\em Mathematics of the USSR-Sbornik}, 63(2):375, 1989.

\bibitem[SC12]{shah2012group}
Parikshit Shah and Venkat Chandrasekaran.
\newblock Group symmetry and covariance regularization.
\newblock In {\em 2012 46th Annual Conference on Information Sciences and
  Systems (CISS)}, pages 1--6. IEEE, 2012.

\bibitem[Sch80]{Schwartz}
J.~T. Schwartz.
\newblock Fast probabilistic algorithms for verification of polynomial
  identities.
\newblock {\em J. Assoc. Comput. Mach.}, 27(4):701--717, 1980.

\bibitem[SPP{\etalchar{+}}05]{sachs2005causal}
Karen Sachs, Omar Perez, Dana Pe'er, Douglas~A Lauffenburger, and Garry~P
  Nolan.
\newblock Causal protein-signaling networks derived from multiparameter
  single-cell data.
\newblock {\em Science}, 308(5721):523--529, 2005.

\bibitem[SRA{\etalchar{+}}21]{shao2021single}
Bin Shao, Jayan Rammohan, Daniel~A Anderson, Nina Alperovich, David Ross, and
  Christopher~A Voigt.
\newblock Single-cell measurement of plasmid copy number and promoter activity.
\newblock {\em Nature communications}, 12(1):1--9, 2021.

\bibitem[SU10]{sturmfels2010multivariate}
Bernd Sturmfels and Caroline Uhler.
\newblock Multivariate {G}aussians, semidefinite matrix completion, and convex
  algebraic geometry.
\newblock {\em Annals of the Institute of Statistical Mathematics},
  62(4):603--638, 2010.

\bibitem[Sul18]{Sullivant}
S.~Sullivant.
\newblock {\em {Algebraic Statistics}}, volume 194 of {\em Graduate Studies in
  Mathematics}.
\newblock AMS, 2018.

\bibitem[TB97]{trefethen1997numerical}
Lloyd~N Trefethen and David Bau.
\newblock {\em Numerical linear algebra}, volume~50.
\newblock Siam, 1997.

\bibitem[Uhl12]{uhler2012geometry}
Caroline Uhler.
\newblock Geometry of maximum likelihood estimation in {G}aussian graphical
  models.
\newblock {\em Annals of Statistics}, 40(1):238--261, 2012.

\bibitem[VAAW16]{vinciotti2016model}
Veronica Vinciotti, Luigi Augugliaro, Antonino Abbruzzo, and Ernst~C Wit.
\newblock Model selection for factorial gaussian graphical models with an
  application to dynamic regulatory networks.
\newblock {\em Statistical applications in genetics and molecular biology},
  15(3):193--212, 2016.

\bibitem[VP90]{verma1990causal}
Thomas Verma and Judea Pearl.
\newblock Causal networks: Semantics and expressiveness.
\newblock In {\em Machine intelligence and pattern recognition}, volume~9,
  pages 69--76. Elsevier, 1990.

\bibitem[WZV{\etalchar{+}}04]{wille2004sparse}
Anja Wille, Philip Zimmermann, Eva Vranov{\'a}, et~al.
\newblock Sparse graphical {G}aussian modeling of the isoprenoid gene network
  in arabidopsis thaliana.
\newblock {\em Genome biology}, 5(11):1--13, 2004.

\end{thebibliography}
 
\addtocontents{toc}{\protect\setcounter{tocdepth}{1}}

\appendix

\section{Connections to stability}
\label{sec:appendix_stability}

We give a characterization of maximum likelihood estimation for RDAG models, via stability under a group action. We generalise the definitions of stability to a set rather than a group. This offers an alternative to Corollary~\ref{cor:the_MLEs} for the set of MLEs.

Fix a set of invertible $m \times m$ matrices $A$, with entries in $\RR$.
Consider some $Y \in \RR^{m \times n}$. 
By analogy to an orbit and stabiliser under a group action, we define the orbit and stabiliser under a {\em set} $A$ to be, respectively,
$$ A \cdot Y := \{ a Y \mid a \in A \}, \qquad A_Y := \{ a \in A \mid aY = Y \}. $$
This allows us to define stability notions analogous to the group situation.

\begin{definition}
\label{defn:stability_setE}
We say that the matrix $Y \in \RR^{m \times n}$, under the set $A$, is

\begin{itemize}
\item[(i)] unstable if there exist $a_n \in A$ with $a_n Y \to 0$ as $n \to \infty$, i.e. $0 \in \overline{A \cdot Y}$
\item[(ii)] semistable if $Y$ is not unstable, i.e. $0 \notin \overline{A \cdot Y}$
\item[(iii)] polystable if $Y \neq 0$ and the set $A \cdot Y$ is (Euclidean) closed 
\item[(iv)] stable if $Y$ is polystable and $A_Y$ is finite. 
\end{itemize}
\end{definition}

The above notions of stability are usually studied for $A$ a reductive group~\cite{Dolgachev, MumfordGIT}. They were studied for reductive and non-reductive groups in~\cite{amendola2020invariant}.  We are not aware of these definitions being used before without any group structure on $A$.

We relate maximum likelihood estimation for an RDAG model to these stability notions. 
For coloured graph $(\Gcal,c)$, we recall the definition of $A(\Gcal,c)$ from~\eqref{eqn:Agc}. Given a set $A \subseteq \GL_m$, we define $A_{\SL} = \{a \in A \mid \det(a) = 1\}$ and $A_{\SL}^{\pm} = \{ a\in A \mid \det(A) = \pm 1 \}$. 

\begin{theorem}\label{thm:RDAGstabilityVsMLE}
Consider an RDAG model on $(\Gcal,c)$ with compatible colouring $c$ and sample matrix $Y \in \RR^{m \times n}$. Then stability under $A(\Gcal,c)_{\SL}$ relates to maximum likelihood estimation:
  \[ \begin{matrix}
        \text{(a)} & Y \text{ unstable} &  \Leftrightarrow &  \ell_Y \text{ unbounded from above} \\ 
    \text{(b)} & Y \text{ semistable} & \Leftrightarrow & \ell_Y \text{ bounded from above} \\
    \text{(c)} & Y \text{ polystable} & \Leftrightarrow & \text{MLE exists} \\
    \text{(d)} & Y \text{ stable}  & \Leftrightarrow & \text{MLE exists uniquely.}
    \end{matrix} \]
\end{theorem}

\begin{proposition}
\label{prop:second-bijection}
Fix the RDAG model on $(\Gcal,c)$ and set $A:= A(\Gcal,c)_{\SL}$. If $\lambda a\T a$ is an MLE given $Y$, where $a \in A$ and $\lambda > 0$ is a scalar, then the set of MLEs given $Y$ are in bijection with $A_Y$ under mapping $b \in A_Y$ to $\lambda (a + b - \id)\T (a + b - \id)$.
\end{proposition}

Theorem~\ref{thm:RDAGstabilityVsMLE} applies to any DAG model, see Remark~\ref{rmk:DAGvsRDAG}. Therefore, Theorem~\ref{thm:RDAGstabilityVsMLE} generalises~\cite[Theorem 5.3]{amendola2020invariant} in multiple ways. First, it extends from \emph{transitive} DAGs\footnote{A DAG is transitive if whenever $k \to j$ and $j \to i$ are in $\mathcal{G}$ then so is $k \to i$. It is exactly the condition to make $A(\Gcal)$ a group, see \cite[Proposition~5.1]{amendola2020invariant}.} to all DAGs. 
Second, it extends from uncoloured DAG models to RDAG models.
Third, it adds part~(d) about stable samples. Moreover, Proposition~\ref{prop:second-bijection} gives a bijection between the MLEs and the stabilising set. 

To prove Theorem~\ref{thm:RDAGstabilityVsMLE} and Proposition~\ref{prop:second-bijection}, we first we generalise \cite[Proposition~3.4 and Theorem~3.6]{amendola2020invariant} to no longer require that $A$ is a group. We say that a set $A$ is closed under non-zero scalar multiples if $a \in A$ implies $ta \in A$ for all $t \in \RR^\times$.

\begin{proposition}\label{prop:GeneralStabilityVsMLE}
Let $A$ be a set of real invertible $m \times m$ matrices, closed under non-zero scalar multiples. 
There is a correspondence between stability under $A_{\SL}^{\pm}$ and maximum likelihood estimation in the model $\Mcal_A$ given sample matrix $Y \in \RR^{m \times n}$: 
$$ \begin{matrix} (a) & Y \text{ unstable}  & \Leftrightarrow & \text{likelihood $\ell_Y$ unbounded from above} \\ 
(b) &  Y \text{ semistable} & \Leftrightarrow & \text{likelihood $\ell_Y$ bounded from above} \\ 
(c) & Y \text{ polystable}  & \Rightarrow & \text{MLE exists.} \end{matrix} $$
The MLEs, if they exist, are the matrices $\lambda a\T a$, where $\| a\cdot Y\| > 0$ is minimal in $A_{\SL}^{\pm} \cdot Y$ and $\lambda \in \RR_{>0}$ is the unique global minimum of
    \[\RR_{>0} \to \RR, \qquad \; x \mapsto \frac{x}{n} \|a \cdot Y\|^2 - m\log(x).\]
If $A$ contains an orthogonal matrix of determinant $-1$, then $A_{\SL}^{\pm}$ can be replaced by~$A_{\SL}$.
\end{proposition}
 
\begin{proof}
This is proved using the same argument as~\cite[Proposition~3.4 and Theorem~3.6]{amendola2020invariant}.
Maximising $\ell_Y$ over $\Mcal_A$ is equivalent to minimising
    \begin{align*}
        f \colon A \to \RR, \; \qquad a \mapsto \frac{1}{n} \|a \cdot Y\|^2 - \log\det(a\T a).
    \end{align*}
We write $a \in A$ as $\tau b$, where $\tau \in \RR_{>0}$ and $b \in A_{\SL}^{\pm}$. Setting $x := \tau^2$ we compute
    \[ f(a) = \frac{\tau^2}{n} \| b \cdot Y \|^2 - \log\det(\tau^2 b\T b)
    = \frac{x}{n} \|b \cdot Y\|^2 - m \log(x).   \]
The infimum of the function $x \mapsto x C - m\log(x)$ increases as $C \geq 0$ increases, hence
    \begin{equation}\label{eq:doubleInf}
        \sup_{a \in A} \ell_Y(a\T a) = - \inf_{a \in A} f(a) = - \inf_{x \in \RR_{>0}} \left( \frac{x}{n} \left( \inf_{b \in A_{\SL}^{\pm}} \|b \cdot Y\|^2 \right) - m \log(x) \right).
    \end{equation}
We have $\inf_{a \in A} f(a) = - \infty$ if and only if $\inf_{b \in A_{\SL}^{\pm}} \|b \cdot Y\|^2 = 0$, i.e. if and only if $Y$ is unstable. This shows parts (a) and (b).

To prove (c) assume that $Y$ is polystable under $A_{\SL}^{\pm}$. Then $C := \inf_{b \in A_{\SL}^{\pm}} \|b \cdot Y\|^2$ is strictly positive, as $Y$ is semistable. Since $A_{\SL}^{\pm} \cdot Y$ is closed in $\RR^{m \times n}$, we see that $C$ is attained in the compact set $(A_{\SL}^{\pm} \cdot Y) \cap \{ Z \in \RR^{m \times n} \mid \|Z\|^2 \leq C+1 \}$. Thus, an MLE given $Y$ exists. If an MLE exists, then the inner and outer infima in~\eqref{eq:doubleInf} are attained, and any MLE has the form in the statement.

If $A$ contains an orthogonal matrix of determinant $-1$, we can write from the outset $a = \tau o b$ with $\tau \in \RR_{>0}$, $b \in A_{\SL}$ and $o$ orthogonal of determinant $\pm 1$. Setting $x := \tau^2$, we derive the same computation for $f(a)$ (now with $b \in A_{\SL}$), since $o$ is orthogonal. The rest of the proof then works with $A_{\SL}$ instead of $A_{\SL}^{\pm}$.
\end{proof}

\begin{remark}\label{rem:A-SL-forRDAG}
Given an RDAG model $\Mcal_{A(\Gcal,c)}$ with compatible colouring, 
we can always apply Proposition~\ref{prop:GeneralStabilityVsMLE} using stability under $A(\Gcal,c)_{\SL}$ instead of the bigger set $A(\Gcal,c)_{\SL}^{\pm}$.
Indeed, if the number of vertices of colour $s$, $\alpha_s$, is even for all $s \in c(I)$, then $A(\Gcal,c)$ only contains matrices of positive determinant, so $A(\Gcal,c)_{\SL}^{\pm} = A(\Gcal,c)_{\SL}$. If $\alpha_s$ is odd for some vertex colour $s$, then $A(\Gcal,c)$ contains an orthogonal matrix with determinant~$-1$.
\end{remark}

Next, we return to the linear independence condition in Theorem~\ref{thm:MLestimationLinDependence}(b).

\begin{lemma}\label{lem:SemistablePolystable}
Consider the RDAG model on $(\Gcal,c)$ where colouring $c$ is compatible,
 and set $A := A(\Gcal,c)_{\SL}$. 
Assume for a non-zero $Y \in \RR^{m \times n}$ that $M_{Y,s}^{(0)} \notin \mathrm{span} \big\lbrace M_{Y,s}^{(1)},\ldots, M_{Y,s}^{(\beta_s)} \big\rbrace$ for all $s \in c(I)$. Then $Y$ is polystable under $A$ and $A \cdot Y$ is Zariski closed.
\end{lemma}

\begin{proof}
The hypotheses in the statement imply that the log-likelihood $\ell_Y$ is bounded from above, by Theorem~\ref{thm:MLestimationLinDependence}(b). Since $A(\Gcal,c)$ is closed under non-zero scalar multiples, $Y$ is semistable under $A = A(\Gcal,c)_{\SL}$ by Proposition~\ref{prop:GeneralStabilityVsMLE}(b) and Remark~\ref{rem:A-SL-forRDAG}.

We now study the orbit $A \cdot Y$. Let $T$ be the set of diagonal matrices in $A$ and $U$ the unipotent matrices in $A$. Then $A = T \cdot U$ by compatibility and, in fact, any $a \in A$ admits a unique decomposition $a = tu$ with $t \in T$, $u \in U$.
For $s \in c(I)$, recall the construction of $M_{Y,s} \in \RR^{(\beta_s + 1)\times \alpha_s n}$ from Definition~\ref{def:MYs}. Setting $V_s := \RR^{1 \times \alpha_s n}$ we can identify $\RR^{m \times n} \cong \bigoplus_s V_s$ such that the rows of vertex colour $s$ belong to $V_s$. By definition of $M_{Y,s}$, and since the $\prc(s)$ partition $c(E)$, the set $U \cdot Y$ is $H := \prod_s H_s$ with
    \begin{align*}
        H_s = \left\lbrace M_{Y,s}^{(0)} + a_{s,1} M_{Y,s}^{(1)} + \ldots + a_{s,\beta_s} M_{Y,s}^{(\beta_s)} \mid a_{s,t} \in \RR \right\rbrace.
    \end{align*}
The affine space $H_s$ is $M_{Y,s}^{(0)} + X_s$, where $X_s := \mathrm{span} \big\lbrace M_{Y,s}^{(1)},\ldots, M_{Y,s}^{(\beta_s)} \big\rbrace$. Since $M_{Y,s}^{(0)} \notin X_s$ for all $s \in c(I)$ by assumption, we have $0 \notin H_s$ and hence $H_s$ has at least codimension one in $V_s$. We define the linear subspace $K_s := \big( \RR M_{Y,s}^{(0)} \big) \oplus X_s$ of $V_s$. 
Since $T$ acts on 
each $V_s$ with the non-zero scalar of colour $s$, we have
    \begin{align*}
        A \cdot Y = T \cdot (U \cdot Y) = T \cdot H = T \cdot \prod_s H_s
        \subseteq \bigoplus_s K_s \subseteq \bigoplus_s V_s.
    \end{align*}
It suffices to show that $A \cdot Y$ is Zariski-closed in $\bigoplus_s K_s$. Each $H_s$ is an affine subspace of $K_s$ with codimension one. Therefore, there exists a linear form $p_s \in K_s^*$ such that $H_s = \mathbb{V}_{K_s}(p_s - 1)$, where $\mathbb{V}(\cdot)$ denotes the vanishing locus.

We finish the proof by showing that $A \cdot Y = \mathbb{V} \big( \prod_s p_s^{\alpha_s} - 1 \big)$ in $\bigoplus_{s} K_s$.
First, given $W = (W_s)_s \in A\cdot Y = T \cdot H$ we can write $W = a \cdot Z$ with $a \in T$ and $Z = (Z_s)_s \in H$. Then
    \[ \Big(\prod_s p_s^{\alpha_s} \Big)(W) = \prod_s p_s(W_s)^{\alpha_s}
    = \prod_s \big(a_s p_s(Z_s) \big)^{\alpha_s} = \prod_s (a_s)^{\alpha_s} = 1 \]
by the choice of $p_s \in K_s^*$ and since $\det(a) = \prod_s a_s^{\alpha_s} = 1$. On the other hand, suppose $W = (W_s)_s \in \mathbb{V} \big( \prod_s p_s^{\alpha_s} - 1 \big) \subseteq \bigoplus_{s} K_s$. Set $a_s := p_s(W_s)$, then we have $\prod_s a_s^{\alpha_s} = 1$, so the $a_s$ define some $a \in T$. Moreover, $W' := (a^{-1}_s W_s)_s \in H$ by definition of the $a_s$ and hence $W = a \cdot W'$ is contained in $T \cdot H = A \cdot Y$.
\end{proof}

\begin{proposition}\label{prop:RDAGStabilityVsLinDependence}
Consider an RDAG model on $(\Gcal,c)$ with compatible colouring $c$ and $A := A(\Gcal,c)_{\SL}$. Let $Y \in \RR^{m \times n}$ be a sample matrix.
Stability under $A$ relates to linear independence conditions on the matrices $M_{Y,s}$:
$$ \begin{matrix} (a) &  Y \text{ unstable}   & \Leftrightarrow &  M_{Y,s}^{(0)} \in \mathrm{span} \big\lbrace M_{Y,s}^{(i)} : i \in [\beta_s] \big\rbrace \text{ for some } s \in c(I) \\ 
(b) & Y \text{ semistable} &  \Leftrightarrow &  Y 
\text{ polystable} \\ 
    (c) & Y \text{ polystable} & \Leftrightarrow &  M_{Y,s}^{(0)} \notin \mathrm{span}  \big\lbrace M_{Y,s}^{(i)} : i \in [\beta_s] \big\rbrace \text{  for all } s \in c(I) \\
(d) & Y \text{ stable} &  \Leftrightarrow &  M_{Y,s} \text{ has full row rank for all } s \in c(I) \end{matrix} $$
\end{proposition} 

\begin{proof}
Proposition~\ref{prop:GeneralStabilityVsMLE} in combination with Theorem~\ref{thm:MLestimationLinDependence} yields part (a) and the forwards direction of (c), while Lemma~\ref{lem:SemistablePolystable} gives the backwards direction of (c).
We obtain part~(b) as a direct consequence of (a) and (c).

For part (d), it suffices to see that a polystable $Y$ has a trivial stabiliser $A_Y$ if and only if for all $s \in c(I)$ the rows $M_{Y,s}^{(1)},\ldots, M_{Y,s}^{(\beta_s)}$ are linearly independent. So let $Y$ be polystable. By construction of $M_{Y,s}$, a matrix $a \in A$ satisfies $aY = Y$ if and only if
    \begin{equation}\label{eq:stableRDAG}
        a_s M_{Y,s}^{(0)} + \sum_{t \in [\beta_s]} a_{s,t} M_{Y,s}^{(t)} = M_{Y,s}^{(0)},
    \end{equation}
    for all $s \in c(I)$,
where $a_s \in \RR^{\times}$ is the entry of $a$ for vertex colour $s$, and $a_{s,t} \in \RR$ is the entry of $a$ for the parent relationship colour encoded by $(s,t)$, where $t \in [\beta_s]$. We note that \eqref{eq:stableRDAG} implies $a_s = 1$ and $\sum_{t \in [\beta_s]} a_{s,t} M_{Y,s}^{(t)} = 0$, by polystability of $Y$ and part~(c).

If $M_{Y,s}^{(1)},\ldots, M_{Y,s}^{(\beta_s)}$ are linearly independent, then \eqref{eq:stableRDAG} has exactly one solution, namely $a_s = 1$ and $a_{s,t}=0$ for all $t \in [\beta_s]$. Thus, if $M_{Y,s}^{(1)},\ldots, M_{Y,s}^{(\beta_s)}$ are linearly independent for all $s \in c(I)$, then $A_Y = \{\id \}$.
On the other hand, if for some $s \in c(I)$ the rows $M_{Y,s}^{(1)},\ldots, M_{Y,s}^{(\beta_s)}$ are linearly dependent, then $\sum_{t \in [\beta_s]} a_{s,t} M_{Y,s}^{(t)} = 0$ has infinitely many solutions. Distinct solutions give distinct unipotent matrices $a \in A$ by using $a_{s,t}$ as the entry for edge colour $t \in \prc(s)$, and setting all other off-diagonal entries of $a$ to zero. Such a unipotent matrix $a \in A$ satisfies $aY = Y$, since the sets $\prc(s)$ are disjoint, so the $a_{s,t}$ do not affect any rows of $Y$ with a different vertex colour. In conclusion, $A_Y$ is infinite if $M_{Y,s}^{(1)},\ldots, M_{Y,s}^{(\beta_s)}$ are linearly dependent for some $s \in c(I)$.
\end{proof}

\begin{proof}[Proof of Theorem~\ref{thm:RDAGstabilityVsMLE}]
Combine Proposition~\ref{prop:RDAGStabilityVsLinDependence} with Theorem~\ref{thm:MLestimationLinDependence}.
\end{proof} 

We now turn to Proposition~\ref{prop:second-bijection}. As above, $A := A(\Gcal,c)_{\SL}$ for an RDAG with compatible colouring. Denote the set of diagonal (respectively unipotent) matrices in $A$ by $T$ (respectively $U$). By compatibility of the colouring, $A = T \cdot U$ and in fact any $a \in A$ admits a unique factorisation $a = tu$ with $t \in T$, $u \in U$.

\begin{lemma}\label{lem:twoStepRDAG}
Consider the RDAG model on $(\Gcal,c)$ where colouring $c$ is compatible. For $A := A(\Gcal,c)_{\SL}$ write $A = T \cdot U$ as above. If $Y \in \RR^{m \times n}$ is polystable under $A$, then the following hold:
    \begin{itemize}
        \item[(a)] $U \cdot Y$ contains a unique element $\widetilde{Y}$ of minimal norm.
        \item[(b)] For $t \in T$ and $u \in U$, $\|tu \cdot Y\| \geq \|t \cdot \widetilde{Y} \|$ with equality if and only if $u \cdot Y = \widetilde{Y}$.
        \item[(c)]  Let $a, \widetilde{a} \in A$ be such that $a \cdot Y$ and $\widetilde{a} \cdot Y$ are of minimal norm in $A \cdot Y$. Then there is an orthogonal $t \in T$ such that $\;ta \cdot Y = \widetilde{a} \cdot Y$.
    \end{itemize}
\end{lemma}

\begin{proof}
For part~(a), we recall that the $\prc(s)$, $s \in c(I)$, partition the edge colours $c(E)$. Therefore, when minimising
    \[ \|uY\|^2 = \sum_{s \in c(I)} \Big\| M_{uY,s}^{(0)} \Big\|^2 = \sum_{s \in c(I)} \Big\| M_{Y,s}^{(0)} + \sum_{t \in [\beta_s]} u_{s,t} M_{Y,s}^{(t)} \Big\|^2 \]
over $u \in U$ we can minimise each summand separately. For each $s \in c(I)$, the affine space $M_{Y,s}^{(0)} + \mathrm{span}\big\lbrace M_{Y,s}^{(t)} \colon t \in [\beta_s] \big\rbrace$ has a unique element of minimal norm, call it $M_s$. Hence, $U \cdot Y$ has a unique element of minimal norm, $\widetilde{Y}$, determined by $M_{\widetilde{Y},s}^{(0)} = M_s$ for all $s \in c(I)$. (Note that there may be several $u \in U$ with $uY = \widetilde{Y}$.)

To prove part~(b), we use (the proof of) part~(a) to obtain
    \begin{equation}\label{eq:twoStepRDAG}
        \big\| M_{tuY,s}^{(0)} \big\|^2 = \big\| t_s \, M_{uY,s}^{(0)} \big\|^2
    = |t_s|\, \big\| M_{uY,s}^{(0)} \big\|^2 \geq |t_s|\, \big\| M_{\widetilde{Y},s}^{(0)} \big\|^2 = \big\| M_{t \widetilde{Y},s}^{(0)} \big\|^2
    \end{equation}
for all $s \in c(I)$, hence $\|tu \cdot Y \| \geq \| t\widetilde{Y} \|$. The latter inequality is strict if and only if there is strict inequality in \eqref{eq:twoStepRDAG} for at least one $s$. By uniqueness of $\widetilde{Y}$, this is the case if and only if $uY \neq \widetilde{Y}$.

For~(c), write $a = tu$ with $t \in T$ and $u \in U$. Since $aY$ is of minimal norm in $A \cdot Y$, we deduce $uY = \widetilde{Y}$ using~(b). Thus, $aY \in T \cdot \widetilde{Y}$ and similarly $\widetilde{a}Y \in T \cdot \widetilde{Y}$. As $T \cdot \widetilde{Y} \subseteq A \cdot Y$ the matrices $aY$ and $\widetilde{a}Y$ are also of minimal norm in $T \cdot \widetilde{Y}$. We note that $T$ is a group isomorphic to the reductive group $\{ (t_s)_{s \in c(I)} \mid t_s \in \RR^{\times}, \; \prod_s t_s^{\alpha_s} = 1\}$. Hence the Kempf-Ness theorem, see \cite[Theorem~2.2]{amendola2020invariant}, for the action of $T$ implies that there is some orthogonal $t \in T$ that relates the minimal norm elements $aY$ and $\widetilde{a}Y$ in $T \cdot \widetilde{Y}$.
\end{proof}

We conclude this appendix with a proof of Proposition~\ref{prop:second-bijection}.

\begin{proof}[Proof of Proposition~\ref{prop:second-bijection}]
Recall that $A = A(\Gcal,c)_{\SL} = T \cdot U$, where $T$ is the diagonal matrices in $A$, and $U$ the unipotent matrices in $A$. If $aY = Y$, then for all $s \in c(I)$,
    \[ M_{Y,s}^{0} = a_s M_{Y,s}^{(0)} + \sum_{t \in [\beta_s]} a_{s,t} M_{Y,s}^{(t)}. \]
We have $M_{Y,s}^{(0)} \notin \mathrm{span} \big\lbrace M_{Y,s}^{(t)} \colon t \in [\beta_s] \big\rbrace$, since $Y$ is polystable. Hence $a_s = 1$ for all $s$, i.e. $a \in U$ and therefore $A_Y = U_Y$. We set $N_Y := U_Y -\id$, which consists of strictly upper triangular matrices. It suffices to show that for fixed MLE $\lambda a\T a$ the map
    \begin{align*}
        \varphi \colon N_Y &\to \{ \text{MLEs given } Y \} \\
        b &\mapsto \lambda (a + b)\T (a + b)
    \end{align*}
is well-defined and bijective. Note that $bY = 0$ for any $b \in N_Y$. Therefore, $(a + b)Y = aY$ is of minimal norm in $A \cdot Y$ and thus $\varphi(b)$ is an MLE by Proposition~\ref{prop:GeneralStabilityVsMLE}.

For surjectivity, let $\lambda \widetilde{a}\T \widetilde{a}$ be another MLE given $Y$. Then $aY$ and $\widetilde{a} Y$ are of minimal norm in $A \cdot Y$, hence there is an orthogonal $t \in T$ with $aY = t\widetilde{a}Y$ by Lemma~\ref{lem:twoStepRDAG}(c). We set $b := t \widetilde{a} - a$ so that $b \cdot Y = 0$ and $(\id + b)Y = Y$. By compatibility of the colouring we have $t \widetilde{a} \in A$ and thus all entries of $b = t \widetilde{a} - a$ obey the colouring $c$. However, $bY = 0$ implies $b_s = 0$ for all $s$ by polystability of $Y$, hence $b \in N_Y$. We compute $\varphi(b) = \lambda (t\widetilde{a})\T (t \widetilde{a}) = \lambda \widetilde{a}\T \widetilde{a}$ using orthogonality of $t$.

For injectivity, let $b, b' \in N_Y$ be such that $\varphi(b) = \varphi(b')$. Let $t \in T$ be defined by $t_s = 1$ if $a_s > 0$ and $t_s = -1$ if $a_s < 0$. Then $t$ is orthogonal and thus
    \[ (ta + tb)\T (ta + tb) = (a+b)\T (a+b).\]
Similarly, $(ta + tb')\T (ta + tb') = (a+b')\T (a+b')$. Then $\varphi(b) = \varphi(b')$ implies
    \begin{equation}\label{eq:bijectionCholesky}
        (ta + tb)\T (ta + tb) = (ta + tb')\T (ta + tb').
    \end{equation}
Moreover, $tb$, $tb'$ are strictly upper triangular and $ta \in A$ has positive diagonal entries, by construction of $t$. Hence, applying uniqueness of the Cholesky decomposition to \eqref{eq:bijectionCholesky} gives $ta + tb = ta + tb'$, and we deduce $b = b'$.
\end{proof}

\section{Connections to Gaussian group models}
\label{sec:connections_GGM}

A model is a Gaussian group model if it is equal to $\Mcal_A$, see~\eqref{eqn:ME}, where the set $A$ a group. In this case, the second term in the log-likelihood~\eqref{eqn:likelihood_e} is the minimisation of the norm over a group orbit. This perspective was used in~\cite{amendola2020invariant} to relate existence of the MLE to notions of stability under a group action. In this appendix, we characterise when the set of matrices $A(\Gcal,c)$ from~\eqref{eqn:Agc} is a group. 
We use Popov's criteron to study stability, and give our third and final description of the set of MLEs in an RDAG model.

\subsection{The butterfly criterion}

Suppose $(\mathcal{G},c)$ is a coloured DAG with compatible $c$. We give necessary and sufficient conditions for $A(\mathcal{G},c)$ to be a group. We call a subset $A \subseteq \GL_m$ {\em linear} if $A = L \cap \GL_m$, where $L$ is a linear subspace of $m \times m$ matrices.

\begin{lemma}\label{lem:GroupOnlyMultiplication}
Let $A\subseteq \GL_{m}$ be a linear subspace of matrices containing the identity matrix. Then $A$ is a group if and only if it is closed under multiplication.
\end{lemma}

\begin{proof}
A group is closed under multiplication. Conversely, if $A$ is closed under multiplication, to be a group it must also be closed under inverses. For a matrix $a \in A$, its characteristic polynomial is $f_a(t) = t^m + c_1 t^{m-1} + \dots + c_m$. We know $c_m \neq 0$ because $c_m$ is (up to sign) the determinant of $a$. This means $a\cdot \frac{1}{-c_m}(a^{m-1} + c_1 a^{m-2} + \dots + c_{m-1} \id) = \id$. 
Suppose $L$ is the linear subspace of matrices such that $A = L \cap \GL_m$. 
Then $a^{-1} = \frac{1}{-c_m}(a^{m-1} + c_1 a^{m-2} + \dots + c_{m-1} \id) \in L \cap \GL_m = A$, since $\id \in A$.
\end{proof}

For a pair of vertices $i,j \in I$, define $b(ij) = \{k\ |\ i \leftarrow k, k \leftarrow j \in \mathcal{G}\}$. Let $\mathcal{G}_{b(ij)}$ denote the coloured subgraph on $\{i\} \cup \{j\} \cup b(ij)$, with edges $i \leftarrow k, k \leftarrow j$ for each $k \in b(ij)$, and colours inherited from $c$. We call $\Gcal_{b(ij)}$ a \emph{butterfly graph}.

\begin{proposition}[The butterfly criterion]
\label{prop:butterfly} 
Consider the RDAG model on $(\Gcal,c)$ where colouring $c$ is compatible. The set $A(\mathcal{G},c)$ is a group if and only if
\begin{itemize}
    \item[(a)] $\mathcal{G}$ is transitive; and
    \item[(b)] if $c(ij) = c(kl)$ for edges $j \to i, \; l \to k$ in $\Gcal$, then $\mathcal{G}_{b(ij)} \simeq \mathcal{G}_{b(kl)}$.
\end{itemize}
\end{proposition}

\begin{proof}
Observe that $A(\mathcal{G},c)$ is a group if and only if it is closed under multiplication, since it is a linear subspace of $\GL_m$ that contains the identity matrix.
Hence, by Lemma~\ref{lem:GroupOnlyMultiplication}, we need to characterize when $A(\mathcal{G},c)$ is closed under multiplication. We have $gh \in A(\mathcal{G},c)$ for $g, h \in A(\mathcal{G},c)$ if and only if
\begin{enumerate}
    \item $(gh)_{ii} = (gh)_{jj}$ whenever $c(i) = c(j)$;
    \item $(gh)_{ij} = (gh)_{kl}$ whenever $j \to i$, $l \to k$ in $\Gcal$ have $c(ij) = c(kl)$; and
    \item $(gh)_{ij} = 0$ whenever $j \not \to i$ in $\mathcal{G}$.
    \end{enumerate} 

For $(1)$, observe that $(gh)_{ii} = g_{ii} h_{ii}$. Thus, if $c(i) = c(j)$ then $(gh)_{ii} = (gh)_{jj}$. 
For (2), take $j \to i$, $l \to k$ in $\Gcal$ with $c(ij) = c(kl)$. Then
    \[ (gh)_{ij} = g_{ii} h_{ij} + g_{ij} h_{jj} + \sum_{p \in b(ij)} g_{ip} h_{pj}
    \quad \text{ and } \quad
    (gh)_{kl} = g_{kk} h_{kl} + g_{kl} h_{ll} + \sum_{q \in b(kl)} g_{kq} h_{ql}, \]
hence $(gh)_{ij} = (gh)_{kl}$ if $\mathcal{G}_{b(ij)} \simeq \mathcal{G}_{b(kl)}$. Conversely, assume $(gh)_{ij} = (gh)_{kl}$ as a polynomial identity in the unknown entries of matrices $g$ and $h$.
By compatibility, $c(i) = c(k)$, so $g_{ii} h_{ij} = g_{kk} h_{kl}$.
Vertex and edge colours are disjoint and the sums over $b(ij)$ and $b(kl)$ only involve edge colours. Thus, $(gh)_{ij} = (gh)_{kl}$ implies $g_{ij} h_{jj} = g_{kl} h_{ll}$, so $h_{jj} = h_{ll}$, and the sum over $b(ij)$ must equal the sum over $b(kl)$. This means $c(j) = c(l)$, and the two collections $(c(ip),c(pj)), p \in b(ij)$ and $(c(kq),c(ql)), q \in b(kl)$ of \emph{ordered} pairs counted with multiplicity agree. Compatibility ensures the correct colours on the vertices in $b(ij)$ and $b(kl)$ as well, hence $\mathcal{G}_{b(ij)} \simeq \mathcal{G}_{b(kl)}$.

For $(3)$, observe that if $j \not \to i$ in $\mathcal{G}$ then $g_{ij} = 0 = h_{ij}$ and therefore $(gh)_{ij} = \sum_{k \in b(ij)} g_{ik} h_{kj}$. The latter is zero for all $g, h \in A(\Gcal,c)$ if and only if $b(ij) = \emptyset$. Thus, condition~$(3)$ is equivalent to the following: if $j \not \to i$ in $\Gcal$, then there does not exist $k \in I$ with $j \to k$ and $k \to i$ in $\Gcal$, i.e. $\mathcal{G}$ must be transitive, by contraposition.
We have shown that $(1),(2)$ and $(3)$ are satisfied if and only if conditions~(a) and (b) hold.
\end{proof}

\begin{example}
Surprisingly, two graphs can have all the same butterfly graphs without being isomorphic. We present an example. 
Consider the following coloured graph with 10 black (square) vertices, and edges that are red (solid), green (squiggly), orange (dashed) or brown (dotted).
\begin{center}
    \begin{tikzcd}[row sep = small, decoration={snake,amplitude=1pt}]
        & \squared{$c_1$} \ar[lddd, Maroon, dotted, bend right = 30] & & \squared{$b_1$} \ar[ll, red] \ar[lldddddd, orange, dashed] \ar[lldddd, OliveGreen, decorate] & \\
        
        & & & & \\
        
       & \squared{$c_2$} \ar[ld, Maroon, dotted] & & \squared{$b_2$} \ar[ll, red] \ar[lluu, orange, dashed] \ar[lldddd, OliveGreen, decorate] & \\
        
        \squared{$d_1$} & & & & \squared{$a_1$} \ar[luuu, Maroon, dotted, bend right = 30] \ar[lu, Maroon, dotted] \ar[ld, Maroon, dotted] \ar[lddd, Maroon, dotted, bend left = 30]  \\

        & \squared{$c_3$} \ar[lu, Maroon, dotted] & & \squared{$b_3$} \ar[ll, red] \ar[lluu, orange, dashed] \ar[lluuuu, OliveGreen, decorate] & \\

         & & & &  \\
        
        & \squared{$c_4$} \ar[luuu, Maroon, dotted, bend left = 30] & & \squared{$b_4$} \ar[ll, red] \ar[lluu, orange, dashed] \ar[lluuuu, OliveGreen, decorate] & 
    \end{tikzcd}
\end{center}
We add some further edges:  four purple edges $a_1 \to c_i$, four blue edges $b_i \to d_1$, and a yellow edge $a_1 \to d_1$.
Now consider the graph obtained by exchanging the green (squiggly) and orange (dashed) edges.

The butterfly graphs for the two graphs are the same, as follows. On the yellow edge, the butterfly graphs both have four paths consisting of a brown edge followed by a blue edge, and four that are a purple edge followed by a brown edge. 
Similarly, we can check the butterfly graphs at the other edge colours.

However, the two coloured graphs are not isomorphic. Indeed, the only way to get an isomorphism is to permute the b-layer and the c-layer. The red (solid) edges give the identity permutation, the orange (dashed) edges give the cycle $\sigma = (1 \; 4 \; 3 \; 2)$, and the green (squiggly) edges give $\sigma^2$. Hence an isomorphism would need to consist of permutations $\mu_1$ and $\mu_2$ of $\{1,2,3,4\}$ with $\mu_1 {\rm id} \mu_2 = {\rm id}$, $\mu_1 \sigma \mu_2 = \sigma^2$, $\mu_1 \sigma^2 \mu_2 = \sigma$. The first condition implies $\mu_2 = \mu_1^{-1}$, hence $\sigma$ and $\sigma^2$ need to be simultaneously conjugate to $\sigma^2$ and $\sigma$ respectively. This implies $(\sigma^2)^2 = \sigma$, a contradiction because $\sigma^4 = {\rm id}$.
\end{example}

\subsection{Popov's criterion}
\label{sec:popov} 

If $A(\Gcal,c)$ is a group, we can prove the important Lemma~\ref{lem:SemistablePolystable} differently, via a criterion of Popov \cite[Theorem~4]{popov1989closed}. The criterion characterises when an orbit under a connected solvable group is closed, provided the underlying field is algebraically closed. Due to the latter assumption, we work in this subsection with RDAG models defined over the complex numbers, and note that many of our results and proofs carry over to the complex case, see Remark~\ref{rem:complexRDAG}. We start by describing Popov's criterion for the group $G := A(\Gcal,c)_{\SL}$.

Since $G \subseteq \GL_m(\CC)$ is a group of invertible upper triangular matrices, it is solvable. We decompose $G = T \cdot U \subseteq \GL_m$ as a semi-direct product, where $T$ is the subgroup of diagonal matrices in $G$, and $U$ is the subgroup of unipotent matrices in $G$. The group $G$ acts on $\CC^{m \times n}$ by left-multiplication. Let $f_{k,l} \in \CC[\CC^{m \times n}]$, $k \in [m], \, l \in [n]$, be the coordinate functions on $\CC^{m \times n}$. Let $x_s$ be the coordinate function corresponding to vertex colour $s \in c(I)$, and let $x_{s,t}$ be the coordinate function for the edge colour $t \in \prc(s)$. Given a tuple of samples $Y \in \CC^{m \times n}$ we consider the orbit map
    \[ \begin{matrix} \mu_{G \cdot Y} \colon & G  & \to & \CC^{m \times n} \\ &  g & \mapsto & g \cdot Y \end{matrix} \qquad \text{or, on coordinate rings,} \qquad \mu_{G \cdot Y}^*(f_{k,l}) = \sum_{j=1}^m x_{c(kj)} Y_{j,l}. \]
We define
    \[ R_Y := \mu_{G \cdot Y}^* \big( \CC[\CC^{m \times n}] \big) = \CC \Big[ \sum_{j=1}^m Y_{j,l} \, x_{c(kj)} \, \Big\vert \, k \in [m], l \in [n] \Big] \subseteq \CC[G]. \]
Since $R_Y$ is a $\CC$-algebra, we obtain the semigroup
    \begin{align*}
        \fX_{G \cdot Y} := \bigg\lbrace (d_s)_{s \in c(I)} \in \fX(T) \, \Big\vert \, \prod_{s \in c(I)} x_{s}^{d_s} \in R_Y \bigg\rbrace,
    \end{align*}
where $\fX(T) \cong \ZZ^{|c(I)|} / \big( \ZZ \cdot (\alpha_s)_{s \in c(I)} \big)$ is the character group of $T$.

\begin{theorem}[{Popov's Criterion,~\cite[Theorem~4]{popov1989closed}}]\label{thm:PopovCriterion}
Let $G$ and $Y$ be as above. The orbit $G \cdot Y$ is Zariski closed if and only if $\fX_{G \cdot Y}$ is a group.
\end{theorem}

\begin{remark}
The group $G = A(\Gcal,c)_{\SL}$ may not be connected as required in \cite[Theorem~4]{popov1989closed}. However, the orbit $G \cdot Y$ is Zariski-closed if and only if $G^{\circ} \cdot Y$ is Zariski-closed, where $G^\circ$ is the identity component of $G$. Thus, after restricting to $G^\circ = T^\circ U$ we may assume that $G$ is connected. Restricting to $T^\circ$ amounts to restricting to the torsion-free part of $\fX(T)$: if $\alpha$ is the greatest common divisor of all $\alpha_s, s \in c(I)$, then $T^\circ \cong \big\lbrace (g_s)_{s \in c(I)} \mid \prod_s g_s^{\alpha_s / \alpha} = 1 \big\rbrace$ and $\, \fX(T^{\circ}) = \ZZ^{|c(I)|} / \big( \ZZ \cdot (\alpha_s/\alpha)_{s \in c(I)} \big)$.
\end{remark}

\begin{proof}[Second Proof of Lemma~\ref{lem:SemistablePolystable}]
The matrix $Y$ is semistable
by Proposition~\ref{prop:GeneralStabilityVsMLE}(b) and Theorem~\ref{thm:MLestimationLinDependence}(b).
Fix $s \in c(I)$ and let $M_{Y,s}^{\dagger}$ be the Hermitian transpose of $M_{Y,s}$. Since $M_{Y,s}^{(0)} \notin \mathrm{span} \big\lbrace M_{Y,s}^{(1)},\ldots,M_{Y,s}^{(\beta_s)} \big\rbrace$ we have $\ker(M_{Y,s}^\dagger) \subseteq \mathrm{span}\{e_1, \ldots, e_{\beta_s}\} \subseteq \CC^{\beta_s + 1}$. Therefore, $e_0$ is in the orthogonal complement of $\ker(M_{Y,s}^{\dagger})$, i.e. in the image of $M_{Y,s}$, so there is some $z \in \CC^{\alpha_s n}$ with $M_{Y,s}z = e_0$. By construction of the matrix $M_{Y,s}$, the equation
    \[ \begin{pmatrix} x_s & x_{s,1} & x_{s,2} & \cdots & x_{s, \beta_s} \end{pmatrix} \big( M_{Y,s}z \big)
    = \begin{pmatrix} x_s & x_{s,1} & x_{s,2} & \cdots & x_{s, \beta_s} \end{pmatrix} e_0 = x_s\]
shows that $x_s$ is a $\CC$-linear combination of the $\sum_{j=1}^m x_{c(kj)} Y_{j,l}$, where $k \in c^{-1}(s)$ and $l \in [n]$; the coefficients are given by $z \in \CC^{\alpha_s n}$. In particular, $x_s \in R_Y$.
Since the coordinate functions $x_s$, $s \in c(I)$ generate the character group $\fX(T)$ (thinking of characters as algebraic group morphisms $G \to \CC^{\times}$), we conclude $\fX_{G \cdot Y} = \fX(T)$. Hence, $\fX_{G \cdot Y}$ is a group and $G \cdot Y$ is Zariski closed by Popov's criterion, Theorem~\ref{thm:PopovCriterion}.
\end{proof}

\subsection{Bijection between the stabiliser and the set of MLEs}

So far we have given two descriptions of the set of MLEs given $Y$ in an RDAG model. Corollary~\ref{cor:the_MLEs} gives a linear space of possible $\Lambda$, while  Proposition~\ref{prop:second-bijection} gives an additive bijection between the MLEs and the stabiliser.
Here we give an alternative (multiplicative) bijection between the set of MLEs and the stabiliser, when $A(\Gcal,c)$ is a group. This is similar to~\cite[Proposition~3.9]{amendola2020invariant}, which gives a multiplicative surjection between the MLEs and the stabiliser, for a
Gaussian group model on a reductive group.

\begin{proposition}\label{prop:StabiliserMLEsGroup}
Consider the RDAG model on $(\Gcal,c)$ where colouring $c$ is compatible and $A(\Gcal,c)$ is a group. Set $A := A(\Gcal,c)_{\SL}$ and let $Y \in \RR^{m \times n}$ be polystable under $A$. Given an MLE $\lambda a\T a$, where $a \in A$ and $\lambda$ is as in Proposition~\ref{prop:GeneralStabilityVsMLE}, we have a bijection
    \begin{align*}
        \varphi \colon A_Y &\to \{ \text{MLEs given } Y \} \\
        g &\mapsto \lambda g\T \! a\T \! a g.
    \end{align*}
\end{proposition}

\begin{proof}
For $g \in A_Y$ we have $a g \cdot Y = aY$, which is of minimal norm in $A \cdot Y$ as $\lambda a\T a$ is an MLE. Hence, $\varphi(g) = \lambda (ag)\T(ag)$ is another MLE given $Y$, by Proposition~\ref{prop:GeneralStabilityVsMLE}, and we see that $\varphi$ is well-defined.
For surjectivity, let $\lambda \widetilde{a}\T \widetilde{a}$ be another MLE given $Y$. Then $aY$ and $\widetilde{a}Y$ are of minimal norm in $A \cdot Y$, hence there is an orthogonal $t \in T$ with $taY = \widetilde{a} Y$ by Lemma~\ref{lem:twoStepRDAG}(c). We obtain $g := a^{-1} t^{-1} \widetilde{a} \in A_Y$, where we crucially used that $A(\Gcal,c)$ (and hence $A$) is a group. 
The orthogonality of $t$ gives $\varphi(g) = \lambda \widetilde{a}\T \widetilde{a}$.

To prove injectivity, let $g, g' \in A_Y$ be such that $\varphi(g) = \varphi(g')$. The latter implies $g\T a\T ag = {g'}\T a\T a g'$, which shows that $h := a g' g^{-1} a^{-1}$ is orthogonal. Therefore, $h$ is diagonal, because any orthogonal upper triangular matrix is diagonal. Moreover, using $g,g' \in A_Y$ we have $haY = aY$, i.e. $h \in A_{aY}$. Note $Y$ and $aY$ have the same orbit (closure), where we again use that $A$ is a group. Thus, $aY$ is polystable as $Y$ is polystable. Combining this with $h(aY) = aY$ and $h$ diagonal implies $h = \id$. Finally, $\id = h = a g' g^{-1} a^{-1}$ shows $g = g'$.
\end{proof}

\bigskip
\bigskip

\noindent
\footnotesize {\bf Authors' addresses:}

\smallskip

\noindent Visu Makam, Institute for Advanced Study, Princeton, and University of Melbourne, \hfill {\tt visu@umich.edu}.

\noindent  Philipp Reichenbach, Technische Universit\"at Berlin, {\tt reichenbach@tu-berlin.de}.

\noindent  Anna Seigal, Harvard University, {\tt aseigal@math.harvard.edu}.

\end{document}